\documentclass[twoside]{amsart}
\usepackage{mathrsfs}
\usepackage{amsfonts}
\usepackage{amssymb}
\usepackage{amssymb}
\usepackage{amssymb}
\usepackage{amssymb}
\usepackage{amssymb}
\usepackage{amssymb}
\usepackage{amssymb}
\usepackage{amssymb}
\usepackage{amsmath}
\usepackage{yhmath}
\usepackage[all]{xy}

\usepackage{lastpage}

\RequirePackage{amsmath} \RequirePackage{amssymb}
\usepackage{amscd,latexsym,amsthm,amsfonts,amssymb,amsmath,amsxtra}
\usepackage[colorlinks=true, urlcolor=blue,bookmarks=true,bookmarksopen=true,
citecolor=blue,hypertex]{hyperref}

    \newcommand{\BC}{{\mathbb {C}}} 
     \newcommand{\BF}{{\mathbb {F}}}

     \newcommand{\sH}{{\mathscr {H}}}

    \newcommand{\CO}{{\mathcal {O}}} \newcommand{\CP}{{\mathcal {P}}}
     
    \newcommand{\CS}{{\mathcal {S}}} 
     
    \newcommand{\CW}{{\mathcal {W}}}

    \newcommand{\RG}{{\mathrm {G}}}

    \newcommand{\RU}{{\mathrm {U}}}

     \newcommand{\bx}{{\bf {x}}} \newcommand{\bm}{{\bf {m}}}\newcommand{\Mat}{{\mathrm {Mat}}}
      \newcommand{\bt}{{\bf {t}}}
     
     \newcommand{\bn}{{\bf {n}}} \newcommand{\bW}{{\bf {W}}}

    \newcommand{\lenth}{{\mathrm {\lenth}}}

    \newcommand{\Gal}{{\mathrm{Gal}}} \newcommand{\GL}{{\mathrm{GL}}}
    \newcommand{\Hom}{{\mathrm{Hom}}} 
    \newcommand{\Ind}{{\mathrm{Ind}}}

    \newcommand{\id}{{\mathrm{id}}}

    \newcommand{\cond}{\mathrm{cond}} 
    \renewcommand{\Re}{{\mathrm{Re}}}

    \newcommand{\Sp}{{\mathrm{Sp}}}
    \newcommand{\diag}{{\mathrm{diag}}}

\newcommand{\Supp}{\mathrm{Supp}}

 \newcommand{\SL}{{\mathrm{SL}}}
 \newcommand{\SO}{{\mathrm{SO}}}

\newcommand{\vol}{{\mathrm{vol}}}

\newcommand{\Mp}{{\mathrm{Mp}}}
\newcommand{\rht}{{\mathrm{ht}}}

    \newcommand{\pair}[1]{\langle {#1} \rangle}
    \newcommand{\wpair}[1]{\left\{{#1}\right\}}

    \newcommand{\incl}{\hookrightarrow}
    
     \newcommand{\ra}{\rightarrow}

    \theoremstyle{plain}

    \newtheorem{thm}{Theorem}[section] \newtheorem{cor}[thm]{Corollary}
    \newtheorem{lem}[thm]{Lemma}  \newtheorem{prop}[thm]{Proposition}
     \newtheorem{defn}[thm]{Definition}

    \numberwithin{equation}{section}

\usepackage[top=1in, bottom=1in, left=1.25in, right=1.25in]{geometry}

\title{Stability of Rankin-Selberg gamma factors for $\Sp(2n), \widetilde {\Sp}(2n)$ and $\RU(n,n)$}

\author{Qing Zhang}
\address{Department of Mathematics, The Ohio State University,
100 Math Tower 231 West 18th Ave,
Columbus OH 43210}
\email{zhang.1649@osu.edu}

\begin{document}

\maketitle

\begin{abstract}
Let $F$ be a $p$-adic field and $E/F$ be a quadratic extension. In this paper, we prove the stability of Rankin-Selberg gamma factors for $\Sp_{2n}(F), \widetilde \Sp_{2n}(F)$ and $\RU_{E/F}(n,n)$ when the characteristic of the residue field of $F$ is not $2$. 
\end{abstract}


\section*{Introduction}
Let $G_n$ be $\Sp_{2n}, \widetilde \Sp_{2n}$ and $\RU_{E/F}(n,n)$, where $E/F$ is a quadratic extension of local or global field. The global Rankin-Selberg zeta integrals for the generic irreducible cuspidal automorphic representations of $G_n$ twisted by generic irreducible cuspidal representations of $\GL_m$ has been developed by Gelbart, Piatetski-Shapiro, Ginzburg, Rallis and Soudry, \cite{GePS2, GiRS1, GiRS2}. Recently, the standard properties of such local $\gamma$-factors were established by Kaplan \cite{Ka}. As a complimentary result of their work, in this paper, we prove the stability of the local gamma factor for a generic representation of $G_n(F)$ when twisted by a sufficiently highly ramified character of $\GL_1$ for a $p$-adic field $F$, when the residue field of $F$ is not 2. More precisely, the main result of this paper is the following
\begin{thm}
Let $F$ be a $p$-adic field such that the characteristic of its residue field is odd, $E/F$ be a quadratic extension. Let $\psi_U$ be a generic character of a maximal unipotent subgroup of $G_n(F)$ defined by a given nontrivial additive character $\psi$ of $F$. Let $\pi_1,\pi_2$ be two $\psi_U$-generic irreducible smooth representations of $G_n(F)$ with the same central character. If $\eta$ is a highly ramified quasi-character of $F^\times$, then $$\gamma(s,\pi_1,\eta,\psi)=\gamma(s,\pi_2,\eta,\psi).$$
\end{thm}
Here the $\gamma$-factors are the Rankin-Selberg gamma factors, see $\S$1 for more details. We also notice that the main theorem also holds for $\RU_{E/F}$ if the residue characteristic of $F$ is 2 and $E/F$ is unramified.

Here we remark that in the $\Sp_{2n}$ case, this result can be deduced from previous work. Cogdell, Kim, Piatetski-Shapiro and Shahidi proved the stability of gamma factors for classical groups (which at least includes $\Sp_{2n}$ and $\SO_n$) in \cite{CKPSS},  where the gamma factors are defined using Langlands-Shahidi method. In \cite{Ka}, Kaplan proved that Rankin-Selberg gamma factors agree with the Langlands-Shahidi gamma factor. Thus our result in the $\Sp_{2n}$ case follows from the stability result in \cite{CKPSS} and Kaplan's result on the agreement of the two type gamma factors. In the $\RU_{E/F}(n,n)$-case, the stability of the Langlands-Shahidi gamma factors is proved in \cite{KK}. Thus in principle, our result in the $\RU_{E/F}(n,n)$ case should follow from an agreement result of the two type $\gamma$-factors, which is unfortunately not included in \cite{Ka}. 

In this paper, we prove the stability of gamma factors for $G_n$ in the Rankin-Selberg context. Although one can deduce this by pulling back the Langlands-Shahidi gamma factors via \cite{Ka}, it is still important to have a proof of stability that remains within the context of integral representations, since there are L-functions that we have integral representations for that are not covered by the Langlands-Shahidi method. So developing methods that work in the integral representation context have an intrinsic value.

Our proof of the stability of gamma factors follows the ideas of Baruch, \cite{Ba1, Ba2} and is based on analysis of partial Bessel functions associated with Howe vectors, which can be viewed as a continuation of the work \cite{Zh1, Zh2}. One main ingredient of the proof is a result of stability properties of partial Bessel functions associated with Howe vectors, see Theorem \ref{thm311}, which might have some independent interest. For example, if a more general form of Theorem \ref{thm311} is true, see the Remark after \ref{thm311}, it is possible to get a local converse theorem for $\Sp_{2n}$ and $\RU(n,n)$, see \cite{Zh1, Zh2} for the local converse theorem for the small rank case. To the author's knowledge, there is no local converse theorem obtained from the Langlands-Shahidi's gamma factors directly. We also expect that the method used here can be used to prove stability results for more groups and gamma factors.

Various results on stability of gamma factors were obtained in different settings, for example, \cite{JS, Ba1, Ba2, CPS, CKPSS, CPSS, CST1} to list a few of them. Usually, the stability of gamma factors is used in conjunction of local global arguments. For example, in the proof of functoriality for classical groups \cite{CKPSS}, the stability of  gamma factors is used to to resolve lack of local Langlands conjecture. In \cite{CST1}, the results on stability of gamma factors for exterior square for $\GL_n$ were used to show that the local Langlands correspondence for $\GL_n$ preserves $\varepsilon$-factors for exterior square and symmetric square.

The paper is organized as follows. In $\S$1, we briefly review the definitions of local zeta integrals and $\gamma$-factors for generic representations of $\Sp_{2n}\times \GL_1$. In $\S$2, we review the concept of Howe vectors following \cite{Ba1} and prove several lemmas which will be used in the later calculations. In $\S$3, we prove a stability result of partial Bessel functions associated with Howe vectors, which is the technical core of the proof of our main theorem. We prove our main theorem in the $\Sp_{2n}$ case in $\S4$, and give a brief account in the $\widetilde \Sp_{2n}$ and $\RU_{E/F}(n,n)$ case in $\S$5 and $\S$6.

\section*{Acknowledgement}
I am very grateful to my advisor Professor Jim Cogdell for many useful suggestions, consistent encouragement and support. I thank Dani Szpruch for his help on metaplectic groups.

\section*{Notations}
Let $F$ be a $p$-adic field, $\CO$ be the ring of integers, $\CP$ be the maximal ideal of $\CO$ and $\varpi$ be a uniformizer of $F$, i.e., a generator of $\CP$. Let $q_F=|\CO/\CP|$, and $|~|_F$ be the standard valuation of $F$ with $|\varpi|_F=q_F^{-1}$.

\subsection*{The symplectic group $\Sp_{2n}$ and its subgroups} Let $n>1$ be an integer and $\Sp_{2n}$ be the rank $n$ symplectic group defined by the matrix $$\begin{pmatrix}& J_n\\ -J_n & \end{pmatrix}, \textrm{ where } J_n=\begin{pmatrix}&&1\\ &\adots & \\ 1&& \end{pmatrix}.$$
Explicitly, $$\Sp_{2n}(F)=\wpair{g\in \GL_{2n}(F): {}^t\! g \begin{pmatrix}& J_n\\ -J_n & \end{pmatrix} g=\begin{pmatrix}& J_n\\ -J_n & \end{pmatrix} }.$$
Let $P=MN$ be the Siegel Levi subgroup of $\Sp_{2n}$, where
$$M=\wpair{\bm_n(g):=\begin{pmatrix}g& \\ & g^* \end{pmatrix},g\in \GL_n(F),g^*=J_n{}^t\! g^{-1}J_n },$$
and $$N=\wpair{\bn_n(X):= \begin{pmatrix}I_n & X\\ & I_n \end{pmatrix}, X\in \textrm{Mat}_{n\times n}(F), {}^t\!X=J_n X J_n}.$$
Let $U_M$ be the upper triangular unipotent subgroup of $M$, and $U=U_MN$, which is the maximal unipotent subgroup of the upper triangular Borel subgroup.

 Let $R$ be the subgroup of the Levi of $P$ which consists elements of the form
$$r(y,x)=\bm_n\begin{pmatrix} I_{n-2} && y\\ &1& x\\ &&1 \end{pmatrix}, y\in \textrm{Mat}_{(n-2)\times 1}(F)\cong F^{n-2}, x\in F.$$

\subsection*{Roots and Weyl group}Denote $$w_1=\begin{pmatrix}& 1 &&\\ I_{n-1} &&&\\ &&&I_{n-1}\\ &&1& \end{pmatrix},$$
and set $j(g)=w_1gw_1^{-1}$, for $g\in \Sp_{2n}$.

Let $T$ be the maximal torus which consists elements of the form $t=\diag(a_1,\dots,a_n, a_n^{-1},\dots, a_1^{-1})$. The simple roots of $\Sp_{2n}$ are $\alpha_i, 1\le i\le n-1, \beta$, where
$$\alpha_i(t)=\frac{a_i}{a_{i+1}}, 1\le i\le n-1, \beta(t)=a_n^2 .$$
Let $\Sigma^+$ be the set of positive roots of $\Sp_{2n}$ and $\Sigma$ be the set of roots of $\Sp_{2n}$. For $\gamma\in \Sigma$, let $U_\gamma$ be the root space of $\gamma$ and let $\bx_{\gamma}:F\ra U_\gamma$ be the corresponding 1-parameter isomorphism. 

Let $\bW$ be the Weyl group of $\Sp_{2n}$. For $\gamma\in \Sigma^+$, let $s_\gamma\in \bW$ be the simple reflection defined by $\gamma$. Then $s_\gamma$ acts on the set $\Sigma$ by $s_\gamma(\gamma')=\gamma'-\pair{\gamma',\gamma^\vee}\gamma$, where $\gamma^\vee$ is the coroot of $\gamma$, and $\pair{\gamma',\gamma^\vee}$ is the natural paring between roots and coroots.

The Weyl group $\bW$ is generated by $s_{\alpha_i}$ and $s_\beta$. We can take representative of $s_{\alpha_i}$, $s_\beta$, by
$$s_{\alpha_i}=\bm_n\begin{pmatrix} 1&&&\\ & \ddots&& \\ &&&1&\\ &&1&& \\ &&&&\ddots & \\ &&&&&1\end{pmatrix}, 1\le i\le n-1,$$
where the block $\begin{pmatrix} &1\\ 1& \end{pmatrix}$ is in the $(i,i+1)\times (i,i+1)$ position,
and $$s_{\beta}=\begin{pmatrix}I_{n-1}&&&\\ &&1&\\ &-1&&\\ &&&I_{n-1} \end{pmatrix}.$$

It is easy to check that $w_1=s_{\alpha_1}s_{\alpha_2}\dots s_{\alpha_{n-1}}$. Let $w_0=w_1s_\beta w_1^{-1}=j(s_\beta)$. In matrix form, we have
$$w_0=\begin{pmatrix} &&1\\ &I_{2n-2} &\\ -1&& \end{pmatrix}.$$

\subsection*{The group $\SL_2$} We will use the following notations for elements of $\SL_2(F)$:
$$\bm_1(a)=\begin{pmatrix}a& \\ & a^{-1} \end{pmatrix},a\in F^\times, \bn_1(b)=\begin{pmatrix}1& b\\ &1 \end{pmatrix}, b\in F, $$
$$ \bar \bn_1(b) =\begin{pmatrix}1& \\ b&1 \end{pmatrix}, b\in F, w^1=\begin{pmatrix} &1\\ -1& \end{pmatrix}.$$
Denote $U^1=\wpair{\bn_1(b),b\in F}$ be the upper triangular unipotent subgroups and $\bar U^1=\wpair{\bar \bn_1(b): b\in F}$ be the lower triangular unipotent subgroups. Let $A=\wpair{\bm_1(a),a\in F^\times}$ be the torus of $\SL_2(F)$.

\subsection*{The metaplectic group $\widetilde \Sp_{2n}$} Let $\widetilde \Sp_{2n}$ be the metaplectic double cover of $\Sp_{2n}$. As a set, we have $\widetilde \Sp_{2n}=\Sp_{2n}\times \mu_2$, where $\mu_2$ is the group $\wpair{\pm1}$. The group multiplication is given by 
$$(g_1,\epsilon_1)\cdot (g_2,\epsilon_2)=(g_1g_2,\epsilon_1\epsilon_2 c(g_1,g_2)), (g_i,\epsilon_i)\in \widetilde \Sp_{2n},$$
where $c(g_1,g_2)$ is the Rao cocycle defined in \cite{Rao}, cf \cite{Sz1} for a brief review of the cocycle formulas.

\section{Local zeta integrals and gamma factors for $\Sp_{2n}$ and $\Mp_{2n}$}
In this section, we review the local zeta integrals for generic representations of $\Sp_{2n}\times \GL_1$ and $\Mp_{2n}\times \GL_1$ defined in \cite{GiRS1}, and the definition $\gamma$-factors. The paper \cite{Ka} contains a nice review of these constructions.

\subsection{Weil representations of $\widetilde \Sp_2\ltimes \sH$} Let $\sH$ be the Heisenberg group of 3 variables, i.e., $\sH=W\oplus F$, where $W$ is the symplectic space of dimension 2, with symplectic structure defined by $\pair{w_1,w_2}=2w_1\begin{pmatrix}&1\\ -1& \end{pmatrix}{}^t\! w_2$\footnote{The factor 2 is added to simplify some formulas.}. Here we view elements of $W$ as row vectors. A typical element of $\sH$ is written as $[w,z]$, for $w\in W,z\in F$. The product in $\sH$ is given by 
$$[w,z]+[w',z']=[w+w',z+z'+\frac{1}{2}\pair{w,w'}].$$
We identify $\Sp(W)$ with $\SL_2(F)$. Recall that $\widetilde \SL_2$ denote the metaplectic double cover of $\SL_2$. For later use, we recall the Rao cocycle and the product in $\widetilde \Sp_2$. For $(g_i, \zeta_i)\in \widetilde \SL_2, i=1,2,$, we have
$$(g_1,\zeta_1)(g_2,\zeta_2)=(g_1 g_2, \zeta_1 \zeta_2 c(g_1,g_2)),$$
where $c:\SL_2(F)\times \SL_2(F)\ra \wpair{\pm1}$ is defined by
$$c(g_1,g_2)=(\bx(g_1),\bx(g_2))_F (-\bx(g_1)\bx(g_2), \bx(g_1g_2))_F,$$
where $$\bx\begin{pmatrix} a& b\\ c& d \end{pmatrix}=\left\{\begin{array}{lll} c, & c\ne 0, \\ d, & c=0. \end{array}
\right.$$
For these formulas, see \cite{Sz2} for example.

For an element $g\in \SL_2(F)$, we sometimes write $(g,1)\in \widetilde \SL_2(F)$ as $g$ by abuse notation.

A representation $\pi$ of $\widetilde \SL_2(F)$ is called genuine if $\pi(\zeta g)=\zeta \pi(g)$ for all $g\in \widetilde \SL_2(F)$ and $\zeta \in \mu_2$. Let $\psi$ be a nontrivial additive character of $F$, there is a Weil representation $\omega_{\psi}$ of $\widetilde \SL_2(F)\ltimes \sH$ on $\CS(F)$, the Bruhat-Schwartz functions on $F$. 
 For $\phi\in \CS(F)$, $\xi\in F$, we have the familiar formulas:
\begin{align*}
\omega_{\psi}([x,x',z])\phi(\xi)&=\psi(z+2\xi x'+x_0x')\phi(\xi+x), [x,x',z]\in \sH,\\
\omega_{\psi}(w^1)\phi(\xi)&=\gamma(\psi)\hat \phi(\xi),\\
\omega_{\psi}(\bn_1(b))\phi(\xi)&=\psi^{-1}(bx^2)\phi(\xi), b\in F\\
\omega_{\psi}(\bm_1(a)) \phi(\xi)&=|a|^{1/2}\frac{\gamma(\psi)}{\gamma(\psi_a)}\phi(a\xi), a\in F^\times,
\end{align*}
and
$$ \omega_{\psi}(\zeta)f(\xi)=\zeta f(\xi), \zeta \in \mu_2. $$
Here $\hat f(x)=\int_F f(y)\psi(2xy)dy$, where $dy$ is normalized so that $(\hat f)^{\hat~}(x)=f(-x)$, $\gamma(\psi)$ is the Weil index and $\psi_a(x)=\psi(ax)$. For these formulas, see \cite{GePS, GiRS2} for example. Note that the formula is affected by the factor 2 in the formula of $\pair{w_1,w_2}$.

Let $\widetilde A$ be the inverse image of the torus
$A\subset \SL_2(F)$ in $\widetilde \SL_2(F)$. The product in
$\widetilde A$ is given by the Hilbert symbol, i.e.,
$$(\bm_1(a),\zeta_1)(\bm_1(b),\zeta_2)=(\bm_1(ab), \zeta_1\zeta_2(a,b)_F),$$
where $(a,b)_F$ is the Hilbert symbol.
 The function
$$\mu_\psi(a)=\frac{\gamma(\psi)}{\gamma(\psi_a)}$$
satisfies
$$\mu_{\psi}(a)\mu_\psi(b)=\mu_\psi(ab)(a,b),$$
and thus defines a genuine character of $\widetilde A$.

\subsection{Genuine induced representation of $\widetilde \SL_2$}
Recall that we denote $U^1$ the upper triangular unipotent of $\SL_2(F)$. Then the Borel subgroup of $\widetilde \SL_2$ is $\tilde A U^1$. Let $\eta$ be a quasi-character of $F$, $s\in \BC$, we consider the genuine induced representation $\tilde I(s,\eta,\psi)=\Ind_{\tilde A U^1}^{\widetilde \SL_2}((\mu_\psi)^{-1} \eta |~|^{s-1/2})$ of $\widetilde \SL_2$. Since $\delta_{\tilde A U^1}((\bm_1(a),\zeta))=|a|^2$, an element $f_s\in \tilde I(s,\eta,\psi)$ satisfies the condition
$$f_s(\bn_1(b)(\bm_1(a), \zeta)g)=\zeta \mu_\psi(a)^{-1}\eta(a)|a|^{s+1/2}f_s(g), b\in F, a\in F^\times, \zeta \in \mu_2, g\in \widetilde \SL_2(F).$$

\subsection{The local zeta integral}
Consider the embedding $\iota:\SL_2\ra \Sp_{2n}$
$$g\mapsto \iota(g)= \begin{pmatrix}I_{n-1}&&\\ &g&\\ &&I_{n-1} \end{pmatrix}.$$
Notice that $\iota\begin{pmatrix} &1\\ -1& \end{pmatrix}=s_\beta$.

Recall that $U$ denote the standard maximal unipotent subgroup of $\Sp_{2n}$. Given a nontrivial additive character $\psi$ of $F$, let $\psi_U$ be the generic character of $U$ defined by 
$$\psi_U((u_{ij}))=\psi_U(\sum_{i=1}^{n}u_{i,i+1}), (u_{ij})\in U.$$

Let $\pi$ be an irreducible $\psi_U$-generic representation of $\Sp_{2n}(F)$, and let $\CW(\pi,\psi_U)$ be the space of $\psi_U$-Whittaker models of $\pi$. Let $\eta$ be a quasi-character of $F^\times$. For $W\in \CW(\pi,\psi_U)$, $\phi\in \CS(F)$ and $f_s\in I(s,\eta,\psi^{-1})$, we consider the local zeta integral
\begin{align*}\Psi(W,\phi,f_s)=\int_{U^1\setminus \SL_2} \int_{F^{n-2}} \int_F W(j(r(y,x) g))(\omega_{\psi^{-1}}(g)\phi)(x)f_s(g)dydxdg.
\end{align*}
\textbf{Remark:} (1) Recall that $j(g)=w_1gw_1^{-1}$ for $g\in \Sp_{2n}$. In the above integral, we do not distinguish $g$ with $\iota(g)$ for $g\in \SL_2(F)$ by abuse of notation. \\
(2) The above local zeta integral was first considered by Ginzburg, Rallis and Soudry in \cite{GiRS1} when $\eta$ is trivial. In fact, Ginzburg, Rallis and Soudry defined a global zeta integral, proved it is Eulerian and did the unramified calculation in \cite{GiRS1}. Later, similar constructions were generalized to the situation $\Sp_{2n}\times \GL_k$ for any $k$ in \cite{GiRS2}.
\begin{prop}
\begin{enumerate}
\item The integral $\Psi(W,\phi,f_s)$ is absolutely convergent when $\Re(s)>>0$, defines a rational function of $q_F^{-s}$.
\item There exists a choice of datum such that $ \Psi(W,\phi,f_s)=1$.
\end{enumerate}
\end{prop}
\begin{proof}
It is not hard to show that the integral on the $y$-part has compact support. Thus the inner integral is absolutely convergent. The assertion of (1) follows from a gauge estimate of $W$, which is standard. We omit the details here. Part (2) is proved in Proposition 1.3, \cite{GiRS1} when $\eta=1$. In our case (when $F$ is $p$-adic), we will give a proof using Howe vectors later.
\end{proof}

Consider the standard intertwining operator $M_s: \tilde I(s,\eta,\psi^{-1})\ra \tilde I (1-s,\eta^{-1},\psi^{-1})$ defined by 
$$M_s(f_s)(\tilde g)=\int_{F}f_s(((w^1)^{-1}\bn_1(b),1) \tilde g)db.$$
It is well-known that the above integral is absolutely convergent when $\Re(s)>>0$ and can be memorphically continued to all $\BC$. 
\begin{prop}\label{prop14}
Given a $\psi_U$-generic representation $\pi$ and a quasi-character $\eta$ on $F^\times$, there is a meromorphic function $\gamma(s,\pi,\eta,\psi)$ such that
$$\Psi(W,\phi, M_s(f_s))=\gamma(s,\pi,\eta)\Psi(W,\phi,f_s),$$
for all $W\in \CW(\pi,\psi_U), \phi\in \CS(F),$ and $f_s\in \tilde I(s,\eta,\psi^{-1})$.
\end{prop}
\begin{proof}
This follows from the uniqueness of Fourier-Jacobi models, \cite{GaGP, Su}. For more details, see \cite{Ka} for example.
\end{proof}
\subsection{Local zeta integrals and gamma factors for generic representation of $\widetilde \Sp_{2n}$} The zeta integrals and gamma factors are defined similarly for $\widetilde \Sp_{2n}$ and we give a brief account of that. Let $\tilde U $ be the preimage of $U$ in $\widetilde \Sp_{2n}$. It is well-known that $\tilde U=U\times \mu_2$ as a group, see $\S$3A of \cite{Sz1} for example. The generic character $\psi_U$ of $U$ extends to a character $\psi_{\tilde U}$ of $\tilde U$ by $\psi_{\tilde U}((u,\epsilon))=\epsilon\psi_U(u)$ for $(u,\epsilon)\in \tilde U$. Let $\pi$ be a genuine irreducible admissible $\psi_{\tilde U}$-generic representation of $\widetilde \Sp_{2n}$. By the main result of \cite{Sz1}, the Whittaker functional of $\pi$ is unique. Let $\CW(\pi,\psi_{\tilde U})$ be the space of $\psi_{\tilde U}$-Whittaker functional of $\pi$. Let $\eta$ be a quasi-character of $F^\times$ and let $I(s,\eta)=\Ind_{AU^1}^{\SL_2(F)}(\eta|~|^{s-1/2} )$ be the induced representation of $\SL_2(F)$. For $W\in \CW(\pi,\psi_{\tilde U})$, $\phi\in \CS(F)$ and $f_s \in I(s,\eta)$, one can consider the local zeta integral
$$\Psi(W,\phi,f_s)=\int_{U^1\setminus \SL_2(F)} \int_{F^{n-2}} \int_F W(j(r(y,x) g))(\omega_{\psi^{-1}}(g)\phi)(x)f_s(g)dydxdg ,$$
where an element $g\in \Sp_{2n}$ is identified with the element $(g,1)\in \widetilde \Sp_{2n}$. Similarly, there exists a gamma factor $\gamma(s,\pi,\eta,\psi)$ which satisfies similar property as in the $\Sp_{2n}$ case. See \cite{Ka} for more details.

\section{Howe vectors for $\Sp_{2n}$}

In this section, we review the definition and basic properties of Howe vectors for $\Sp_{2n}$ following \cite{Ba1}, and give some preliminary results which will be used in the proof of the stability of gamma factors.

Let $m>0$ be a positive integer and $K_m=(I_{2n}+\Mat_{2n\times 2n}(\CP^m))\cap \Sp_{2n}(F)$ be the standard congruence subgroup of $\Sp_{2n}(F)$. Let $\psi$ be a fixed additive character of $F$ with conductor $\CO_F$. Consider the character $\tau_m$ of $K_m$ defined by
$$\tau_m((k_{ij}))=\psi(\varpi^{-2m} (\sum_{i=1}^{n} k_{i,i+1})).$$
It is easy to check that $\tau_m$ is indeed a character. Let 
$$d_m=\diag(\varpi^{-m(2n-1)}, \varpi^{-m(2n-3)},\dots, \varpi^{-m}, \varpi^m, \dots, \varpi^{m(2n-1)})\in \Sp_{2n}(F)$$
and $H_m=d_mK_m d_m^{-1}$. Define a character $\psi_m$ on $H_m$ by $\psi_m(h)=\tau_m(d_m^{-1}h d_m), h\in H_m$. For a subgroup $S$ of $\Sp_{2n}$, we will denote $S_m:=S\cap H_m$.

\begin{lem}\label{lemma21}
\begin{enumerate}
\item The two characters $\psi_U$ and $\psi_m$ agree on $U_m=U\cap H_m$. 
\item For a positive root $\gamma$ of $\Sp_{2n}$, then $$U_{\gamma,m}=\wpair{\bx_{\gamma}(r): r\in \CP^{-(2\rht(\gamma)-1)m}},$$
and $$U_{-\gamma,m}=\wpair{\bx_{-\gamma}(r): r\in \CP^{(2\rht(\gamma)+1)m }}.$$
Moreover, we have $$U_m=\prod_{\gamma\in \Sigma^+}U_{\gamma,m},$$
where the product on the right side can be taken in any fixed order of $\Sigma^+$.
\end{enumerate}
\end{lem}
\begin{proof}
One can check (1) and the first part of (2) by direct calculation. The ``moreover" part of (2) comes from the corresponding statement for $K_m\cap U$, cf \cite{St}.
\end{proof}

Let $(\pi,V_\pi)$ be an irreducible smooth $\psi_U$-generic representation of $\Sp_{2n}(F)$. We fix a Whittaker functional $\lambda_\pi\in \Hom_U(\pi,\psi_U)$ and consider the Whittaker functions defined by $\lambda_\pi$, i.e., $W_v(g)=\lambda_\pi(\pi(g)v)$ for $v\in V_\pi$. We will write the identity element $I_{2n}\in \Sp_{2n}(F)$ as 1 for simplicity. We fix a vector $v\in V_\pi$ such that $\lambda_\pi(v)=W_v(1)=1$, and consider the vector
$$v_m=\frac{1}{\vol(U_m)}\int_{U_m}\psi_m(u)^{-1}\pi(u)vdu.$$
Let $C=C(v)$ be an integer such that $v$ is fixed by $\pi(K_C)$ (i.e., $C$ is bigger than the conductor of $v$), then a vector $v_m$ with $m\ge C$ is called a \textbf{Howe vector} as in \cite{Ba1, Ba2}.

\begin{lem}\label{lemma22}
We have
\begin{enumerate}
\item $W_{v_m}(1)=1;$
\item if $m\ge C$, then $\pi(h)v_m=\psi_m(h)v_m$, for all $h\in H_m;$
\item for $k\le m$, we have 
$$v_m=\frac{1}{\vol(U_m)} \int_{U_m}\psi_m^{-1}(u) \pi(u)v_k du.$$
\end{enumerate}
\end{lem}
\begin{proof}
Only (2) needs some work. The key ingredient of the proof of (2) is the Iwahori decomposition of $K_m$ and hence of $J_m$. The details can be found is Lemma 3.2 \cite{Ba1}, or Lemma 5.2 \cite{Ba2} in the $\RU(2,1)$ case. Baruch's thesis \cite{Ba1} is not published, but the proof in our case is the same as the proof in the $\RU(2,1)$ case which is given in \cite{Ba2}.
\end{proof}
By (2) of Lemma \ref{lemma22}, for $m\ge C$, the partial Bessel function $W_{v_m}(g)$ satisfies the relation
\begin{equation}\label{eq21}W_{v_m}(u gh)=\psi_U(u)\psi_m(h)W_{v_m}(g), \forall u\in U, h\in H_m, g\in \Sp_{2n}(F). \end{equation}

\begin{lem}{\label{lemma23}}
For $m\ge C$ and $t\in T$, if $W_{v_m}(t)\ne 0$, then $\alpha_i(t)\in 1+\CP^m$ for all $i$ with $1\le i \le n-1$ and $\beta(t)\in 1+\CP^m$.
\end{lem}
\begin{proof}
Write $\gamma$ for a general simple root. Take an element $r\in \CP^{-m}$. We have the relation
$$t\bx_\gamma(r)= \bx_\gamma( \gamma(t)r) t.$$
Since $\bx_{\gamma}(r)\in U_m\subset H_m$, then by Eq.(\ref{eq21}), we have 
$$\psi_m(\bx_{\gamma}(r))W_{v_m}(t)=\psi_U(\bx_\gamma(\gamma(t)r)) W_{v_m}(t),$$
for $m\ge C$.
Thus if $W_{v_m}(t)\ne 0$, we get $\psi_m(\bx_{\gamma}(r))=\psi_U(\bx_{\gamma}(\gamma(t)r))$, or $\psi(r)=\psi(\gamma(t)r)$ for all $r\in \CP^{-m}$. Since $\psi$ has conductor $\CO_F$, we get $\gamma(t)-1\in \CP^m$, or $\gamma(t)\in 1+\CP^m$. This proves the Lemma.
\end{proof}

\begin{lem}{\label{lemma24}}
Suppose that the residue characteristic of $F$ is not $2$, then the square map $1+\CP^m\ra 1+\CP^m$ is well-defined and surjective.
\end{lem}
\begin{proof}
This is a simple application of Newton's Lemma, Proposition 2, Chapter II of \cite{Lg}. We omit the details. 
\end{proof}

Note that the center $Z$ of $\Sp_{2n}(F)$ is $\wpair{\pm I_{2n}}$. We will write the identity matrix $I_{2n}\in \Sp_{2n}$ as $1$ for simplicity.
\begin{cor}\label{cor25} Suppose the residue characteristic of $F$ is not $2$. Let $t=\diag(a_1,\dots,a_n,a_n^{-1},\dots, a_1^{-1})\in T$, and $m\ge C$, then
$$W_{v_m}=\left\{\begin{array}{lll}\omega_\pi(e), & \textrm{ if } t=e\cdot \diag(a_1',\dots,(a_1')^{-1}), \textrm{ for }a_i'\in 1+\CP^m, e=\pm1; \\ 0, & \textrm{ otherwise,} \end{array} \right.$$
where $\omega_\pi$ is the central character of $\pi$.
\end{cor}
\begin{proof}
Suppose that $W_{v_m}(t)\ne 0$, then $a_i/a_{i+1}\in 1+\CP^m$ for $1\le i\le n-1$ and $a_n^2\in 1+\CP^m$ by Lemma \ref{lemma23}. By Lemma \ref{lemma24}, we can find an element $a_n'\in 1+\CP^m$ such that $(a_n')^2=a_n^2$. Thus $a_n=ea_n$ for some $e\in \wpair{\pm 1}$. Since $a_{n-1}/a_n\in 1+\CP^m$, we can write $a_{n-1}=ea_{n-1}'$ for  $a_{n-1}'=a_n' \cdot \frac{a_{n-1}}{a_n} \in 1+\CP^m$. Inductively, we have $a_i=ea_i'$ for some $a_i'\in 1+\CP^m$. Then $$t=e\cdot \diag(a_1', \dots, a_{n'}, (a_n')^{-1},\dots, (a_1')^{-1}), a_i'\in 1+\CP^m,$$
where we don't distinguish $e$ and $\diag(e,\dots,e)$ by abuse of notation. Since $\diag(a_1',\dots, (a_1')^{-1})\in H_m$, then by Lemma \ref{lemma22} or Eq.(\ref{eq21}), we get
$$W_{v_m}(t)=W_{v_m}(e)=\omega_\pi(e)W_{v_m}(1)=\omega_\pi(e).$$
This completes the proof.
\end{proof}

For $a\in F^\times$, denote $\bt(a)=\diag(a,1,1,\dots,1,1,a^{-1})$.
\begin{lem}{\label{lemma26}}
For $a\in F^\times $ and $y={}^t\!(y_1,\dots,y_{n-2})\in \Mat_{(n-2)\times 1}(F)$, then for $m\ge C$, we have
 $$W_{v_m}(\bt(a)j(r(y,0)))=\left\{ \begin{array}{lll}W_{v_m}(\bt(a)), & \textrm{ if }y_i\in \CP^{(2i+1)m}, \textrm{ for all }i, 1\le i\le n-2, \\ 0, & \textrm{otherwise.} \end{array}\right.$$ 
\end{lem}
\begin{proof}
We have 
\begin{equation} \label{eq22}\bt(a)j(r(y,0))=\bm_n\begin{pmatrix} a & &\\ y& I_{n-2} & \\ &&1 \end{pmatrix}.\end{equation}
If $y_i\in \CP^{(2i+1)m}$, then $j(r(y,0))\in H_m\cap \bar U$, and thus $W_{v_m}( \bt(a)j(r(y,0)))=W_{v_m}(\bt(a))$ by Lemma \ref{lemma22}, or Eq.(\ref{eq21}) for $m\ge C$. Now we suppose $y_i\notin \CP^{(2i+1)m}$ for some $i$. Let $i$ be the biggest integer with this property, i.e., $i$ satisfies $1\le i\le n-2$, $y_i\notin \CP^{(2i+1)m}$ and $y_j\in \CP^{(2j+1)m}$ for all $j$ with $n-2\ge j>i$. By Eq.(\ref{eq21}) and Eq.(\ref{eq22}), we have 
\begin{equation}\label{eq23}W_{v_m}(\bt(a)j(r(y,0)))=W_{v_m}\left( \bm_n\begin{pmatrix} a&&\\ y^i & I_i & \\ &&I_{n-i-1}\end{pmatrix}\right),\end{equation}
where $y^i={}^t\!(y_1,\dots, y_i)\in \Mat_{i\times 1}(F)$. Take $r\in \CP^{-(2i+1)m}$ so that 
$$X(r):=\bx_{\alpha_1+\dots \alpha_i+\alpha_{i+1}}(r)=\bm_n\begin{pmatrix} I_n+re_{1,i+2} \end{pmatrix}\in H_m,$$
where $e_{1,i+2}$ is the $n\times n$ matrix with 1 in the $(1,i+2)$ position, and zero elsewhere.  We have the relation
\begin{equation} \label{eq24}X(-r) \bm_n\begin{pmatrix} a&&\\ y^i & I_i & \\ &&I_{n-i-1}\end{pmatrix}X(r)=\bm_n\begin{pmatrix}1&&&\\ &I_i & y^i r&\\ &&1&\\ &&&I_{n-i-2}  \end{pmatrix} \bm_n\begin{pmatrix} a&&\\ y^i & I_i & \\ &&I_{n-i-1}\end{pmatrix}.\end{equation}
Note that $\psi_m(X(r))=\psi_U(X(r))=1$, and 
$$ \psi_U \left( \bm_n\begin{pmatrix}1&&&\\ &I_i & y^i r&\\ &&1&\\ &&&I_{n-i-2}  \end{pmatrix} \right)=\psi(y_ir).$$
From Eq.(\ref{eq21}), Eq.(\ref{eq23}) and Eq.(\ref{eq24}), we get
$$W_{v_m}(\bt(a) j(r(y,0)))=\psi(y_i r)W_{v_m}(\bt(a) j(r(y,0))). $$
Since $y_i\notin \CP^{(2i+1)m}$, we can take $r\in \CP^{-(2i+1)m}$ such that $\psi(y_ir)\ne 1$. Thus $W_{v_m}(\bt(a)j(r(y,0)))=0$. This completes the proof.
\end{proof}

\section{A stability property of partial Bessel functions}
In this section, we prove a stability property of partial Bessel functions associated with Howe vectors (Theorem \ref{thm311} and Corollary \ref{cor313}), which is the key ingredient of the proof of the stability of the gamma factors.
\subsection{Weyl elements and root spaces}
We recall some notations on the roots and Weyl element of $\Sp_{2n}$. For $t=\diag(a_1,a_2,\dots, a_n,  a_n^{-1},\dots, a_1^{-1})$, the simple roots are given by $$\alpha_i(t)=\frac{a_i}{a_{i+1}}, 1\le i \le n-1, \beta(t)=a_n^2.$$
The Weyl group $\bW$ is generated by $s_{\alpha_i}$ and $s_\beta$. 

We recall the notion of Bruhat order on $\bW$. For $w\in \bW$ with a minimal expression $s_{\xi_1}\dots s_{\xi_l}$ where $\xi_i$ are simple roots. We say that $w'\le w$ if $w'$ can be written as $w'=s_{\xi_{t_1}}\dots s_{\xi_{t_k}}$ with $t_1,\dots t_k \in \wpair{1,2,\dots, l}$ and $t_1<t_2<\dots$, i.e., $w'$ can be written as sub-expression of a minimal expression of $w$, see \cite{Hu}. The definition of the Bruhat order does not depend on the choice of minimal expression of $w$. We will say that $w'<w$ if $w'\le w$ and $w'\ne w$.

Let $\Sigma^+$ be the set of positive roots. Let $U_M$ be the upper triangular unipotent of $\GL_n(F)$ and $N$ be the Siegel unipotent of $\Sp_{2n}(F)$. We will say a positive root $\gamma$ is in $U_M$ or $N$, if the root space of $\gamma$ is in $M$ or $N$. Suppose that $\gamma\in U_M$, then the root space of $\gamma$ is in the $(i,j)$-position, with $1\le i<j\le n$, and the root $\gamma$ is $\sum_{t=i}^{j-1}(\alpha_t)$. If $\gamma\in N$, then the root space of $\gamma$ is in the $(i,k)$ position with $1\le i\le n <k\le 2n$. By symmetry of the root spaces in $N$, we can assume that $i+k\le 2n+1$. If $i+k=2n+1$, i.e., the root space of $\gamma$ is in the skew diagonal, then $\gamma=2\sum_{t=i}^{n-1}(\alpha_t)+\beta$. If $i+k<2n+1$, we put $j=2n+1-i$, then $\gamma=\sum_{t=i}^{j-1}(\alpha_t)+2\sum_{t=j}^{n-1}(\alpha_t)+\beta$.

 \begin{lem}{\label{lemma31}}
Let $\gamma_1,\gamma_2\in \Sigma^+$, $\gamma_1\ne \gamma_2$, $\frac{1}{2}\rht(\gamma_2)<\rht(\gamma_1)\le \rht(\gamma_2) $ and $\pair{\gamma_2,\gamma_1^\vee }=2$, then there exists integers $i,j$, with $1\le i<j \le n$ such that $\gamma_2=2\alpha_{i}+2\alpha_{i+1}+\dots +2\alpha_{n-1}+\beta$ and $\gamma_1=\alpha_{i}+\alpha_{i+1}+\dots +\alpha_{j-1}+2\alpha_{j}+2\alpha_{j+1}+\dots +2\alpha_{n-1}+\beta$.
\end{lem}
We will call a pair $(\gamma_1,\gamma_2)$ of positive roots which satisfies the condition of Lemma \ref{lemma31} a bad pair of positive roots. For the reason that such a pair is ``bad", see Lemma \ref{lemma34}.
\begin{proof}
We first consider $\gamma_1\in U_M$, say, the root space of $\gamma_1$ is in the $(i,j)$ position with $1\le i<j\le n$. Then $\gamma_1^\vee(t)=\diag(1,\dots, 1,t, 1,\dots, t^{-1},1,\dots 1, t, \dots,1,t^{-1},1,\dots,1)$ with $t$ in the $i$ and $2n+1-j$ position and $t^{-1}$ in the $j$ and $2n+1-j$ position. The only positive root $\gamma_2$ with $\gamma_2\ne \gamma_1$ and $\pair{\gamma_2,\gamma_1^\vee}=2$ has root space in the $ (i,2n+1-i)$ position, i.e., $\gamma_2=2(\alpha_i+\dots+\alpha_{n-1})+\beta$. It is clear that $\rht(\gamma_2)=2(n-i)+1>2(j-i)=2\rht(\gamma_1)$. Thus the pair $(\gamma_1,\gamma_2)$ does not satisfy the condition. 

Next, consider the case $\gamma_1\in N$, say, the root space of $\gamma_1$ is in the $(i,k)$ position with $ 1\le i \le n$, $n+1\le k \le 2n$ and $i+k\le 2n+1$. If $i+k=2n+1$, then $\gamma_1^\vee(t)=\diag(1,\dots,1,t,1,\dots,1,\dots,1,t^{-1},1,\dots,1)$, with $t$ in the $i$ position and $t^{-1}$ in the $j=2n+1-i$ position. There is no $\gamma_2$ other than $\gamma_1$ itself such that $\pair{\gamma_2,\gamma_1^\vee}=1$. If $i+k< 2n+1$, let $j=2n+1-k$. Then $i<j\le n$ and we have $\alpha_1=\alpha_{i}+\dots \alpha_{i-1}+2(\alpha_j+\dots+\alpha_{n-1})+\beta$. On the other hand $\gamma_1^\vee(t)$ is the diagonal element with $t$ in the $i$ and $j$ position, $t^{-1}$ in the $k=2n+1-j$ and $2n+1-i$ position. There are two $\gamma_2\ne \gamma_1$ with $\pair{\gamma_2,\gamma_1^\vee}=2$: i.e., $2(\sum_{l=i}^{n-1}\alpha_l)+\beta$ and $2(\sum_{l=j}^{n-1}\alpha_l)+\beta$. The second one has height $2(n-j)+1$ which is smaller than the height of $\gamma_1$. Thus $\gamma_2=2(\sum_{l=i}^{n-1}\alpha_l)+\beta$. This finishes the proof.
\end{proof}
For $\gamma\in \Sigma^+$, recall that we have an element $s_\gamma\in \bW$ which acts on $\Sigma^+$ by $s_\gamma(\gamma')=\gamma'-\pair{\gamma', \gamma^\vee }\gamma$.

Let $w_0=s_{2(\alpha_1+\dots+\alpha_{n-1})+\beta}$. It is not hard to check that $w_0=s_{\alpha_1}\dots s_{\alpha_{n-1}}s_\beta s_{\alpha_{n-1}}\dots s_{\alpha_1}$, which is a minimal expression of $w_0$. In general, we have $s_{2(\alpha_{i}+\dots+\alpha_{n-1})+\beta}=s_{\alpha_i}\dots s_{\alpha_{n-1}}s_\beta s_{\alpha_{n-1}}\dots s_{\alpha_i}$, which is also a minimal expression.

We will say that a Weyl element $w$ is in $M$, if it has a representative in $M$, i.e., $w$ does not involve $s_\beta$. 
\begin{lem} \label{lemma32}
If $w\in M$ and $\gamma\in N$, then $w(\gamma)\in N$, in particular $w(\gamma)>0$. 
\end{lem}
\begin{proof}
This follows from the fact that $M$ normalizes $N$.
\end{proof}

\begin{prop}\label{prop33}
Given a bad pair of positive roots $(\gamma_1,\gamma_2)=(\alpha_{i}+\dots+\alpha_{j-1}+2(\alpha_j+\dots +\alpha_{n-1})+\beta, 2(\alpha_i+\dots+\alpha_{n-1})+\beta)$ for $1\le i< j\le n$ as in Lemma $\ref{lemma31}$. Assume $w\le w_0$, $w(\gamma_1)<0$ and $w(\gamma_2)<0$, then $w$ can be written as the form
$$w=w_1' s_{\alpha_{j-1}}s_{\alpha_{j}}\dots s_{\alpha_{n-1}} s_{\beta}s_{\alpha_{n-1}}\dots s_{\alpha_{i+1}}s_{\alpha_i} w_2', $$
with  $w_1'\le s_{\alpha_1}\dots s_{\alpha_{j-2}}$ and $w_2'\le s_{\alpha_{i-2}}s_{\alpha_{i-3}}\dots s_{\alpha_1} $. 
\end{prop}
Here are some examples of $w$ with $w\le w_0, w(\gamma_1)<0$ and $w(\gamma_2)<0$.
\begin{enumerate}
\item Suppose that $n=3$, $i=1,j=2$, i.e., $(\gamma_1,\gamma_2)=( \alpha_1+2\alpha_2+\beta, 2(\alpha_1+\alpha_2)+\beta)$, then the only $w\le w_0$ which satisfies $w(\gamma_1)<0$ and $w(\gamma_2)<0$ is $w_0=s_{\alpha_1}s_{\alpha_2}s_\beta s_{\alpha_2}s_{\alpha_1}$ itself.
\item Suppose that $n=3$, $i=1,j=3$, i.e., $(\gamma_1,\gamma_2)=(\alpha_1+\alpha_2+\beta, 2(\alpha_1+\alpha_2)+\beta)$,  then $w=s_{\alpha_1}s_{\alpha_2}s_\beta s_{\alpha_2}s_{\alpha_1} =w_0$ or $w=s_{\alpha_2}s_\beta s_{\alpha_2}s_{\alpha_1}$. 
\item Suppose that $n=3$, $i=2,j=3$, i.e., $(\gamma_1, \gamma_2)=(\alpha_2+\beta,2\alpha_2+\beta)$, then $w=s_{\alpha_1}s_{\alpha_2}s_\beta s_{\alpha_2}$ or $s_{\alpha_2}s_\beta s_{\alpha_2}$.
\end{enumerate}

\begin{proof}

As a preparation, we compute the action of $s_{\alpha_k}$ ($1\le k\le n-1$) on the positive roots $\beta, 2\alpha_{n-1}+\beta,\dots, 2(\alpha_1+\dots+\alpha_{n-1})+\beta$ which lies in the skew diagonal of $N$. For simplicity, denote $\beta_i=2(\alpha_i+\dots +\alpha_{n-1})+\beta$ for $1\le i\le n-1$, and $\beta_n=\beta$ temporarily. We have 
$$\pair{\beta_i,\alpha_k^\vee}=\left\{\begin{array}{lll} 2, & i=k,\\ -2, & i=k+1,\\ 0, &\textrm{otherwise.} \end{array}\right.$$
Thus 
\begin{equation} \label{eq31} s_{\alpha_k}(\beta_i)=\beta_i-\pair{\beta_i,\alpha_k^\vee}\alpha_k=\left\{\begin{array}{lll} \beta_{i+1}, & i=k,\\ \beta_{i-1}, & i=k+1,\\ \beta_i, &\textrm{otherwise.} \end{array}\right. \end{equation}
In particular, $s_{\alpha_k}$ preserves the set $\wpair{\beta_i}_{1\le i \le n}.$ We also have 
\begin{equation}\label{eq32} s_\beta(\beta_i)=\beta_i, 1\le i \le n-1.\end{equation}

Now we start the proof. Take a $w\le w_0$ such that $w(\gamma_1)<0$ and $w(\gamma_2)<0$. 

First $w$ must involve $s_\beta$. In fact, if $w$ does not involve $s_\beta$, i.e., $w=s_{\alpha_{m_1}\dots s_{\alpha_{m_t}}}\in M$, then $w(\gamma_1)>0$ and $w(\gamma_2)>0$ from the above fact.

Since $w\le w_0$, we can assume that $w= s_{\alpha_{m_1}}\dots s_{\alpha_{m_t}}s_\beta s_{\alpha_{l_1}}\dots s_{\alpha_{l_k}}$ for $m_1<m_2>\dots <m_{t}, l_k<l_{k-1}\dots <l_1$. In our case, $\gamma_2=\beta_i$ for $1\le i\le n-1$. Suppose that $w(\gamma_1)<0$ and $w(\gamma_2)<0$. We will prove $w$ has the given form by the following claims. 

Claim 1: We have $ s_{\alpha_{l_1}}\dots s_{l_k}(\beta_i)=\beta_n=\beta$. 

By Eq.(\ref{eq31}), the expression $s_{\alpha_{l_1}}\dots s_{l_k}$ preserves the set $\wpair{\beta_i}_{1\le i\le n}$. Thus we can assume $ s_{\alpha_{l_1}}\dots s_{l_k}(\beta_i)=\beta_p$ for some $p$ with $1\le p\le n$. If $p\ne n$, then $ s_\beta s_{\alpha_{l_1}}\dots s_{l_k}(\beta_i)=s_{\beta}(\beta_p)=\beta_p$ by Eq.(\ref{eq32}), and thus $w(\beta_i)=s_{\alpha_{m_1}}\dots s_{\alpha_{m_t}}(\beta_p)>0$ by Lemma \ref{lemma32}. Contradiction. This proves Claim 1.

Claim 2: If $i>1$, then the expression $s_{\alpha_{l_1}}\dots s_{l_k}$ does not involve $s_{\alpha_{i-1}}$. 

By contradiction, we assume that $s_{\alpha_{l_1}}\dots s_{l_k}=s_{\alpha_{l_1}}\dots s_{\alpha_{l_{p}}}s_{\alpha_{i-1}}s_{\alpha_{l_{p+2}}}\dots s_{\alpha_{l_k}} $, with $l_1>\dots l_p> i-1>l_{p+2}>\dots >l_k$. By Eq.(\ref{eq31}), we have 
\begin{align*}
&s_{\alpha_{l_1}}\dots s_{\alpha_{l_{p}}}s_{\alpha_{i-1}}s_{\alpha_{l_{p+2}}}\dots s_{\alpha_{l_k}} (\beta_i)\\
=&s_{\alpha_{l_1}}\dots s_{\alpha_{l_{p}}}s_{\alpha_{i-1}}(\beta_i)\\
=&s_{\alpha_{l_1}}\dots s_{\alpha_{l_{p}}}(\beta_{i-1})\\
=& \beta_{i-1}\ne \beta_n. 
\end{align*}
 This contradicts Claim 1. This proves Claim 2.

Claim 3: The expression $s_{\alpha_{n-1}}\dots s_{\alpha_i}$ must be a sub-expression of $ s_{\alpha_{l_1}}\dots s_{l_k}$.

Write  $ s_{\alpha_{l_1}}\dots s_{l_k}(\beta_i)=s_{\alpha_{l_1}}\dots s_{\alpha_{l_p}} s_{\alpha_{l_{p+1}}}\dots s_{\alpha_{l_k}}$ for $l_1> \dots > l_p\ge i >i-1> l_{p+1}>\dots >l_k$. We have 
$$ s_{\alpha_{l_1}}\dots s_{l_k}(\beta_i)=s_{\alpha_{l_1}}\dots s_{l_p}(\beta_i).$$
If $l_p>i$, then $s_{\alpha_{l_1}}\dots s_{l_p}(\beta_i)=\beta_i $ from the above calculation. Thus $l_p=i$. By induction, we have $ s_{\alpha_{l_1}}\dots s_{l_p}=s_{\alpha_{n-1}}\dots s_{\alpha_{i+1}}s_{\alpha_i}$. This proves Claim 3. 

From the above 3 claims, we get $w=s_{\alpha_{m_1}}\dots s_{\alpha_{m_t}}s_\beta s_{\alpha_{n-1}}\dots s_{\alpha_{i+1}}s_{\alpha_i} w_2'$ for some $w_2'\le s_{\alpha_{i-2}}\dots s_{\alpha_1}$. We consider $s_{\alpha_{n-1}}\dots s_{\alpha_{i+1}}s_{\alpha_i} w_2'(\gamma_1)$. 

Claim 4: We have $s_\beta s_{\alpha_{n-1}}\dots s_{\alpha_{i+1}}s_{\alpha_i} w_2'(\gamma_1)=\alpha_{j-1}+\alpha_j+\dots +\alpha_{n-2}+\alpha_{n-1}$.

We have $s_{\alpha_k}(\alpha_{k-1})=\alpha_{k-1}+\alpha_k$ and $s_{\alpha_k}(\alpha_{k+1})=\alpha_k+\alpha_{k+1}$. Thus
\begin{align*}
& s_{\alpha_{n-1}}\dots s_{\alpha_{i+1}}s_{\alpha_i} w_2'(\alpha_i+\alpha_{i+1}+\dots \alpha_{j-1})\\
=&s_{\alpha_{n-1}}\dots s_{\alpha_{i+1}}s_{\alpha_i}(\alpha_i+\alpha_{i+1}+\dots \alpha_{j-1})\\
=&s_{\alpha_{n-1}}\dots s_{\alpha_{i+1}}(\alpha_{i+1}+\dots \alpha_{j-1})\\
=&\dots\\
=& s_{\alpha_{n-1}}\dots s_{\alpha_{j-1}}(\alpha_{j-1})\\
=&s_{\alpha_{n-1}}\dots s_{\alpha_{j}}(-\alpha_{j-1})\\
=&-(\alpha_{j-1}+\dots +\alpha_{n-2} +\alpha_{n-1}).
\end{align*}
From Claim 1, we have $ s_{\alpha_{n-1}}\dots s_{\alpha_{i+1}}s_{\alpha_i} w_2'(\gamma_2)=\beta$. Thus
\begin{align*}
 &s_{\alpha_{n-1}}\dots s_{\alpha_{i+1}}s_{\alpha_i} w_2'(\gamma_1)\\
=& s_{\alpha_{n-1}}\dots s_{\alpha_{i+1}}s_{\alpha_i} w_2'(\gamma_2-(\alpha_i+\dots \alpha_{j-1}))\\
=& \beta +\alpha_{j-1}+\dots +\alpha_{n-2}+\alpha_{n-1}. 
\end{align*}
Since $s_\beta$ preserves $\alpha_k$ for $k\le n-2$ and $s_\beta(\alpha_{n-1})=\alpha_{n-1}+\beta$, we get
\begin{align*}
& s_\beta s_{\alpha_{n-1}}\dots s_{\alpha_{i+1}}s_{\alpha_i} w_2'(\gamma_1)\\
=&s_\beta( \beta +\alpha_{j-1}+\dots +\alpha_{n-2}+\alpha_{n-1})\\
=& \alpha_{j-1}+\dots +\alpha_{n-1}.
\end{align*}
This proves Claim 4. 

Claim 5: The expression $s_{\alpha_{j-1}}s_{\alpha_j}\dots s_{\alpha_{n-1}}$ must be a sub-expression of $s_{\alpha_{m_1} }\dots s_{\alpha_{m_t}}$.

 By Claim 4, $w(\gamma_1)=s_{\alpha_{m_1} }\dots s_{\alpha_{m_t}}(\alpha_{j-1}+\dots + \alpha_{n-1}) $. If $m_t\ne n-1$, then $ s_{\alpha_{m_1} }\dots s_{\alpha_{m_t}}(\alpha_{j-1}+\dots + \alpha_{n-1})$ must be a sum of $\alpha_{n-1}$ with another root, which cannot be negative. Thus $m_t=n-1$. By induction, we get Claim 5. 

Now the proposition follows from the above Claims.
\end{proof}

We call a tripe $(\gamma_1,\gamma_2,w)$ a bad tripe, if $(\gamma_1,\gamma_2)$ is a bad pair and $w\in \bW$ such that $w(\gamma_1)<0$ and $w(\gamma_2)<0$.

\begin{lem}\label{lemma34}
For $\gamma_1,\gamma_2\in \Sigma^+$ with $\gamma_1\ne \gamma_2$, and $\rht(\gamma_1)\le \rht(\gamma_2)$. If $(\gamma_1,\gamma_2)$ is not a bad pair as in Lemma $\ref{lemma31}$, then $s_{\gamma_1}(\gamma_2)=\gamma_2-\pair{\gamma_2, \gamma_1^\vee}\gamma_1\in \Sigma^+$.
\end{lem}
\begin{proof}
We have $ \pair{\gamma_2, \gamma_1^\vee}=0,\pm 1, \pm 2$. If $\pair{\gamma_2, \gamma_1^\vee}\le 0$, the assertion is clear. If $\pair{\gamma_2, \gamma_1^\vee}=1$, then $\rht(s_{\gamma_1}(\gamma_2))=\rht(\gamma_2)-\rht(\gamma_1)\ge 0$, and thus $s_{\gamma_1}(\gamma_2)>0$. In fact, if we had $s_{\gamma_1}(\gamma_2)<0$, we would have $\rht(s_{\gamma_1}(\gamma_2))<0$.  Now suppose that $\pair{\gamma_2, \gamma_1^\vee}=2$. Since the pair $(\gamma_1,\gamma_2)$ is not bad, we get $\rht(\gamma_1)\le \frac{1}{2}\rht(\gamma_2)$, and thus $$\rht(s_{\gamma_1}(\gamma_2))=\rht(\gamma_2)-2\rht(\gamma_1)\ge 0.$$
The same argument as above shows that $s_{\gamma_1}(\gamma_2)>0$.
\end{proof}

\begin{lem}\label{lemma35}
Given a bad tripe $(\gamma_1,\gamma_2,w)$. We assume that $(\gamma_1,\gamma_2)=(\sum_{t=i}^{j-1}\alpha_t+ 2\sum_{t=j}^{n-1} \alpha_t +\beta, 2\sum_{t=i}^{n-1}\alpha_t+\beta)$ with $1\le i <j \le n$, and $w=w_1'\sigma w_2'$ with $w_1'\le s_{\alpha_1}\dots s_{\alpha_{j-2}}$, $w_2'\le s_{\alpha_{i-2}}\dots s_{\alpha_1}$ and $\sigma =s_{\alpha_{j-1}}\dots s_{\alpha_{n-1}}s_\beta s_{\alpha_{n-1}}\dots s_{\alpha_i}$. See Lemma $\ref{lemma31} $ and Proposition $\ref{prop33}$.
\begin{enumerate}
\item Given $\gamma\in \Sigma^+$ such that $\rht(\gamma_1)\le \rht(\gamma)<\rht(\gamma_2)$. If $s_{\gamma_2}(\gamma)<0$, then there exists an integer $p$ with $i< p\le j$ such that 
$$\gamma=\sum_{t=i}^{p-1} \alpha_t +2\sum_{t=p}^{n-1}\alpha_t +\beta.$$
Moreover, for such a $\gamma$, if $\sigma(\gamma)<0$, then $\gamma=\gamma_1$.
\item Let $\gamma$ be a positive root such that $\rht(\gamma)\ge \rht(\gamma_1)$ and $s_{\gamma_2}(\gamma)>0$, then $\sigma(\gamma)>0$.
\end{enumerate}
\end{lem}
\begin{proof}
(1) By the formular $s_{\gamma_2}(\gamma)=\gamma-\pair{\gamma,\gamma_2^\vee}\gamma_2$, we need to consider the pair $\pair{\gamma,\gamma_2^\vee}$. We have $\gamma_2^\vee(t)=\diag(1,\dots,1,t,1,\dots,1,t^{-1},1,\dots,1)$, where $t$ is in the $i$ position and $t^{-1}$ in the $2n+1-i$ position. For $\gamma\in U_M$, suppose that $\gamma$ is in the $(k,l)$ position, with $1\le k <l \le n$. Since $n-k\ge l-k=\rht(\gamma)\ge \rht(\gamma_1)=2n-i-j+1\ge n-i+1$, we get $k\le i-1$. Thus we have $\pair{\gamma, \gamma_2^\vee}=0$ or $-1$, and hence $s_{\gamma_2}(\gamma)>0$.

Next, we consider $\gamma\in N$. Then it is easy to see that $\pair{\gamma,\gamma_2^\vee}=0,\pm 1$. If $\pair{\gamma,\gamma_2^\vee}=0$ or $-1$, then $s_{\gamma_2}(\gamma)>0$. Thus we need to consider the $\gamma$ with $\pair{\gamma,\gamma_2^\vee}=1$. The root space of such $\gamma$ must be in the $i$-th row or $(2n+1-i)$-th column. In the latter case, we have $\rht(\gamma)\ge \rht(\gamma_2)$. In the former case, by the height condition of $\gamma$, we can choose a $p$ with $i<p\le j$ such that 
$$\gamma=\sum_{t=i}^{p-1} \alpha_t +2\sum_{t=p}^{n-1}\alpha_t +\beta.$$
This proves the first assertion of (1). To prove the moreover part, notice that we can write $\sigma=s_{\alpha_{j-2}}\dots s_{\alpha_i}s_{\gamma_2}$. Given a $\gamma=\sum_{t=i}^{p-1} \alpha_t +2\sum_{t=p}^{n-1}\alpha_t +\beta$ with $i<p\le j$, we have $s_{\gamma_2}(\gamma)=\gamma-\gamma_2=-(\alpha_i+\dots \alpha_{p-1})$. If $p<j$, or equivalently, $p-1\le j-2$, we have $\sigma(\gamma)=-s_{\alpha_{j-2}}\dots s_{\alpha_i}(\alpha_i+\dots \alpha_{p-1})>0$. This proves the moreover part of (1). 

(2) Suppose that there is a $\gamma$ such that $s_{\gamma_2}(\gamma)>0$ but $\sigma(\gamma)<0$. Denote $\xi=s_{\gamma_2}(\gamma)\in \Sigma^+$. Then $\sigma(\gamma)=s_{\alpha_{j-2}}\dots s_{\alpha_i}(\xi)<0$. Thus $\xi\in \Sigma^-_{s_{\alpha_{j-2}} \dots s_{\alpha_i}}$. It is not hard to check that $$ \Sigma^-_{s_{\alpha_{j-2}}\dots s_{\alpha_i}}=\wpair{\sum_{t=i}^{p}\alpha_t, i\le p \le j-2}.$$
Thus we can suppose that $\xi=\sum_{t=i}^p \alpha_t$ for some $p$ with $i\le p\le j-2$. Note that $\rht(\xi)=p+1-i<\rht(\gamma_1)$. By the definition of $\xi$, we have 
$$\xi=s_{\gamma_2}(\gamma)=\gamma-\pair{\gamma,\gamma_2^\vee}\gamma_2.$$
If $\pair{\gamma,\gamma_2^\vee}<0$, then $\xi=\gamma+\gamma_2$ or $\xi=\gamma+2\gamma_2$, which contradicts to $\rht(\xi)<\rht(\gamma_1)$. If $\pair{\gamma,\gamma_2^\vee}=0$, then $\gamma=\xi$, and thus $\rht(\gamma)<\rht(\gamma_1)$. If $\pair{\gamma,\gamma_2}>0$, then $\gamma=\gamma_2+\xi$ or $\gamma=2\gamma_2+\xi$. Note that neither $\gamma_2+\xi$ nor $2\gamma_2+\xi$ is a root because it contains $3\alpha_i$. This proves (2).
\end{proof}

\begin{prop}{\label{prop36}}
Let $w\in \bW$ and $w\le w_0$. Let $\xi_1,\dots, \xi_k \in \Sigma_w^-$ ordered by $\rht(\xi_t)\le \rht(\xi_{t+1})$. 
\begin{enumerate}
\item Suppose that there is no $l$ with $2\le l \le k$ such that $(\xi_1,\xi_l)$ is a bad pair as described in Lemma $\ref{lemma31}$, i.e., $\frac{1}{2}\rht(\beta_l)<\rht(\beta_1)\le \rht(\beta_l)$, then for all $t\in T, r_i\in F,r_1\ne 0$, we have
$$g:=tw\bx_{\xi_k}(r_k) \dots \bx_{\xi_1}(r_1)\bx_{-\xi_1}(-r_1^{-1})\in Bw'B$$
for some $w'$ with $w'<w$. 
\item Suppose that there exists an $l$ with $2\le l \le k$ such that $(\xi_1,\xi_l)$ is a bad pair as described in Lemma $\ref{lemma31}$, we have
$$g:=tw \bx_{\xi_k}(r_k)\dots \hat\bx_{\xi_l}(r_l)\dots \bx_{\xi_1}(r_1) \bx_{\xi_l}(r_l)\bx_{-\xi_l}(-r_l^{-1})\in Bw'B,$$
for some $w'<w$, where $\hat \bx_{\xi_l}(r_l)$ means the term $\bx_{\xi_l}(r_l)$ is omitted.
\end{enumerate}
\end{prop}
\begin{proof}
For any $\gamma\in \Sigma^+$ and $r\in F^\times$, we have 
\begin{equation}\label{eq33} \bx_{\gamma}(r)\bx_{-\gamma}(-r^{-1})\in s_\gamma B.\end{equation}
This follows from a standard Chevalley relation, see \cite{ St}. We will also use the following fact on Bruhat order: given $w',w\in \bW$, then
\begin{enumerate}
\item[$(*)$] $w'<w$ if and only if there exists positive roots $ \xi_1,\dots, \xi_k$ such that $w'=ws_{\xi_1}\dots s_{\xi_k}$ and $ws_{\xi_1}\dots s_{\xi_i}(\xi_{i+1})$ is negative for all $i$ with $1\le i\le k-1$.
\end{enumerate}
For a proof of this fact, see \cite{Hu} for example.

(1) We have $ \bx_{\xi_1}(r_1)\bx_{-\xi_1}(-r_1^{-1})\in s_{\xi_1} B$ from the above Chevalley relation, Eq.(\ref{eq33}). By Lemma \ref{lemma34}, we have $s_{\xi_1}(\xi_t)>0$ for all $t$ with $2\le t\le k$. Thus
$$ tw\bx_{\xi_k}(r_k) \dots \bx_{\xi_1}(r_1)\bx_{-\xi_1}(-r_1^{-1})\in tws_{\xi_1}U_{s_{\xi_1}(\xi_k)} \dots U_{s_{\xi_1}(\beta_2)}B\subset Bw'B.$$
where $w'=ws_{\xi_1}$.  Since $w(\xi_1)<0$, we have $w'<w$, by the fact $(*)$.

(2) We suppose that $(\xi_1,\xi_l)=(\alpha_i+\dots+\alpha_{j-1}+2(\alpha_{j}+\dots+\alpha_{n-1})+\beta, 2(\alpha_i+\dots+\alpha_{n-1})+\beta)$. By Proposition \ref{prop33}, we can assume $$w=w_1'\sigma w_2',$$
for $w_1'\le s_{\alpha_1}\dots s_{\alpha_{j-2}},$ $\sigma=s_{\alpha_{j-1}}\dots s_{\alpha_{n-1}}s_\beta s_{\alpha_{n-1}}\dots s_{\alpha_i}$ and $w_2'\le s_{\alpha_{i-2}}\dots s_{\alpha_1}$. Note that $w_2'$ commutes with $\sigma$. In fact, each $s_{\alpha_p}$ with $p\le i-2$ commutes with each element $s_{\alpha_i}, \dots, s_{\alpha_{n-1}}, s_\beta$. Then we can write 
$$w=w_1' w_2' \sigma.$$
From the fact that $\bx_{\xi_l}(r_l)\bx_{-\xi_l}(-r_l^{-1})=s_{\xi_l}b=\sigma^{-1}s_{\alpha_{j-2}}\dots s_{\alpha_i}b $ for some $b\in B$, see Eq.(\ref{eq33}), we get 
\begin{align*}
g=&tw \bx_{\xi_k}(r_k)\dots \hat \xi_{\beta_l}(r_l)\dots \bx_{\xi_1}(r_1) \bx_{\xi_l}(r_l)\bx_{-\xi_l}(-r_l^{-1})\\
=&tw_1' w_2'   \sigma \bx_{\xi_k}(r_k)\dots \hat \bx_{\xi_l}(r_l)\dots \bx_{\xi_1}(r_1) \sigma^{-1} s_{\alpha_{j-2}}\dots s_{\alpha_i} b\\
=& tw_1'w_2' \bx_{\sigma(\xi_k)}(r_k)\dots \hat\bx_{\sigma(\xi_l)}(r_l) \dots \bx_{\sigma(\xi_1)}(r_1)s_{\alpha_{j-2}}\dots s_{\alpha_i}b.
\end{align*}
From Lemma \ref{lemma34} and Lemma \ref{lemma35}, we get $\sigma(\xi_i)>0$ for all $i>1$. Moreover, we have
$$\sigma(\xi_1)=s_{\alpha_{j-2}}\dots s_{\alpha_i}(s_{\xi_l}(\xi_1))=-s_{\alpha_{j-2}}\dots s_{\alpha_i}(\alpha_i+\dots +\alpha_{j-1})=-\alpha_{j-1}.$$
Thus we get $$g\in Bw_1'w_2' B U_{-\alpha_{j-1}} s_{\alpha_{j-2}}\dots s_{\alpha_{i}}B. $$
From the relation Eq.\ref{eq33}, we get $U_{-\alpha_{j-1}}=Bs_{\alpha_{j-1}}B$, and thus
\begin{equation}\label{eq3.4}g\in Bw_1'w_2'Bs_{\alpha_{j-1}}Bs_{\alpha_{j-2}}\dots s_{\alpha_i}B.\end{equation}

To proceed, we quote a general result from the structure theory of Chevellay group:
\begin{enumerate}
\item[$(**)$] For $w\in \bW$ and $\gamma$ a simple root, we have 
$$\begin{array}{lll} BwBs_\gamma B=Bws_\gamma B , & \textrm{ if }l(ws_\gamma)=l(w)+1,\\
BwBs_\gamma B = B wB \cup Bws_\gamma B, &\textrm{ if } l(ws_\gamma)=l(w)-1 .\end{array}$$
\end{enumerate}
For a proof of this result, see Lemma 25 of \cite{St} for example. 

From \ref{eq3.4} and the fact $(**)$, it is clear that
$$g\in Bw_1' s_{\alpha_{j-1}}s_{\alpha_{j-2}}\dots s_{\alpha_i}w_2'B.$$
The assertion follows from the obvious relation $w':=w_1's_{\alpha_{j-1}}\dots s_{\alpha_i}w_2'<w$.
\end{proof}

\begin{lem}{\label{lemma37}}
Given a pair $(\gamma_1,\gamma_2)=(\sum_{t=i}^{j-1}\alpha_t+ 2\sum_{t=j}^{n-1} \alpha_t +\beta, 2\sum_{t=i}^{n-1}\alpha_t+\beta)$ with $1\le i <j \le n$ as in Lemma $\ref{lemma31}$, and an element $w\in \bW$ such that $w\le w_0$, $w(\gamma_1)<0$ and $w(\gamma_2)<0$. Let $\xi\in \Sigma^+$ such that $\rht(\gamma_1)\le \rht (\xi)\le \rht (\gamma_2)$. If $\gamma_2-\xi=\sum_{t}\delta_t$ is a sum of positive roots $\delta_t$, then there is at least one $t$ such that $w(\delta_t)<0$.
\end{lem}
\begin{proof}
We first claim that $\xi\in N$. In fact, if $\xi\in U_M$, say the root space of $\xi$ is in the $(k,l)$-position with $1\le k<l\le n$, then $\xi=\sum_{t=k}^{l-1}(\alpha_t)$. Since $\rht(\xi)=l-k\ge \rht(\gamma_1)=2n-i-j-1\ge n-i+1$, we get $k\le i-1$, (see the proof of Lemma \ref{lemma35}). Thus in the expression $\gamma_2-\xi$, we have the term $-\alpha_k$. Thus $\gamma_2-\xi$ cannot be a sum of positive roots. This proves the claim.

Then we can suppose that $ \xi=\sum_{t=m}^{p-1}\alpha_t +2\sum_{t=p}^{n-1}\alpha_t +\beta$, for some integers $m,p$ with $1\le m \le p \le n$. Since $\gamma_2-\xi$ is a sum of positive roots, we get $m\ge i$. Thus we have 
$$\gamma_2-\xi=2\sum_{t=i}^{m-1}\alpha_i+\sum_{t=m}^{p-1}\alpha_t.$$
Let $\delta_t$ be the root which involves $\alpha_i$. Then there exists a $q$ with $p-1\ge q\ge i$ such that $\delta_t=\sum_{t=i}^{q}\alpha_t$. 
 By Proposition \ref{prop33}, we can assume that
$$w=w_1' s_{\alpha_{j-1}}\dots s_{\alpha_{n-1}}s_\beta s_{\alpha_{n-1}}\dots s_{\alpha_i} w_2',$$
with $w_1'\le s_{\alpha_1}\dots s_{\alpha_{j-2}},$ and $w_2'\le s_{\alpha_{i-2}}\dots s_{\alpha_1}$. We can get $w(\delta_t)<0$ by a simple calculation. In fact, we have
$$ s_\beta s_{\alpha_{n-1}}\dots s_{\alpha_i} w_2'(\delta_t)=-(\alpha_{q}+\dots \alpha_{n-1}+\beta), $$
and $w_1' s_{\alpha_{j-1}}\dots s_{\alpha_{n-1}}(-(\alpha_{q}+\dots \alpha_{n-1}+\beta))<0 $ by Lemma \ref{lemma32}.
\end{proof}

\begin{lem}[Chevalley relations]\label{lemma38}
For $r,s\in F$ and $\gamma_1,\gamma_2\in \Sigma^+$, we have
$$[\bx_{\gamma_1}(r),\bx_{\gamma_2}(s)]=\prod_{i\ge 1, j\ge 1}\bx_{i\gamma_1+j\gamma_2}(c_{ij}rs),$$
for some $c_{ij}\in F.$
\end{lem}

For $w\in \bW$, define $U_w^+=\wpair{u\in U: wuw^{-1}\in U}$ and $U_w^-=\wpair{u\in U: wuw^{-1}\notin U}.$ Then
$$ U_w^+=\prod_{\gamma\in \Sigma_w^+}U_\gamma, \textrm{ and } U_w^-=\prod_{\gamma\in \Sigma_w^-}U_\gamma,$$
where $\Sigma_w^+=\wpair{\gamma\in \Sigma^+, w(\gamma)>0}$ and $\Sigma_w^-=\wpair{\gamma\in \Sigma^+, w(\gamma)<0}$. Given $w\in \bW$, suppose that $\Sigma_w^-=\wpair{\xi_1,\dots, \xi_l}$. It is well-known that $l=l(w)$, the length of $w$. We now assume $w\le w_0$. Inspired by Proposition \ref{prop36}, we give an order of the finite set $\Sigma_w^-$ as follows. 
\begin{defn}[Order of $\Sigma_w^-$] \label{defn39} We order the set $\Sigma_w^-=\wpair{\xi_1,\xi_2,\dots, \xi_l}$ as follows.
\begin{enumerate}
\item Suppose that there is no pair $(\gamma_1,\gamma_2)$ of positive roots as in Lemma $\ref{lemma31}$, such that $\gamma_1,\gamma_2\in \Sigma_w^-$, we ordered the set $\wpair{\xi_i}$ by $\rht(\xi_i)\le \rht(\xi_{i+1})$. If two roots $\xi,\xi'\in \Sigma^-_w$ have the same height, we do not mind what the order of $\xi$ and $\xi'$ is.
\item Suppose that there exists a pair $(\gamma_1,\gamma_2)$ of positive roots as in Lemma $\ref{lemma31}$ such that $\gamma_1,\gamma_2\in \Sigma_w^-$, we first order the set $\Sigma_w^--\wpair{\gamma_2}$ by height as in $(1)$, then we let $\gamma_2$ be the previous one adjacent to $\gamma_1$, i.e., if $\gamma_1=\xi_i$, then $\gamma_2=\xi_{i-1}$.
\end{enumerate}
\end{defn}

A general element in $U_w^-$ then can be written as $\bx_{\xi_l}(r_l)\dots \bx_{\xi_1}(r_1)$. Recall that the notation $U^-_{w,m}$ means $U^-_w\cap H_m$, which is also $U^-_w\cap U_m$.

\begin{lem}\label{lemma310}
Given $u_w^-=\bx_{\xi_l}(r_l)\dots \bx_{\xi_1}(r_1)\in U_w^--U_{w,m}^-$. Let $q$ be an integer with $1\le q\le l$ such that $\bx_{\xi_k}(r_k)\in U_m$ for $k<q$ but $\bx_{\xi_q}(r_q)\notin U_m$. Let $u=\prod_{\gamma \in \Sigma^+}\bx_\gamma(s_\gamma)\in U_m$. Then for $t\in T$, we have
$$tw\bx_{\xi_l}(r_l)\dots \bx_{\xi_q}(r_q)u=\tilde u t w \bx_{\xi_l}(\tilde r_l) \dots \bx_{\xi_1}(\tilde r_1),$$
for some $\tilde u\in U, \tilde r_t\in F, 1\le t\le l$ , with $|\tilde r_q|=|r_q|$.
\end{lem}
\begin{proof} 
We only consider the case that there exists a pair $(\gamma_1,\gamma_2)$ as in Lemma \ref{lemma31} such that $\gamma_1,\gamma_2\in \Sigma^-_w$, and when $\xi_q=\gamma_2$. Actually one can check from the following argument that the proof of the remaining cases are easier than this exceptional case.

 By our order on $\Sigma_w^-$, we have $\xi_{q+1}=\gamma_1$ and $\rht(\xi_k)\ge \rht(\xi_{q+1})>\frac{1}{2}\rht(\xi_k)$ for $k\ge q$. We can write $u=u^+ \bx_{\xi_{l}}(s_l)\dots \bx_{\xi_1}(s_1)$, for $s_q\in F$ with $\bx_{\xi_k}(s_q)\in U_{\xi_k,m}$ for each $k$, see the ``moreover" part of Lemma \ref{lemma21} (2).

Claim 1: there exists $u_1^+\in U_w^+, u_1^-\in \prod_{t>q}U_{\xi_t}$ such that
$$\bx_{\xi_l}(r_l)\dots \bx_{\xi_q}(r_q) u^+ \bx_{\xi_q}(-r_q)\dots \bx_{ \xi_l}(-r_l)=u_1^+ u_1^-. $$
The idea is that we move $u^+$ to the left side step by step using Chevalley relations, Lemma \ref{lemma38}. In each step, a commutator element will come out. By Lemma \ref{lemma37}, the commutator does not involve elements in $U_{\xi_q}$. We provide more details now.
Write $\Sigma_w^+=\wpair{\delta_1,\dots, \delta_v}$. For $k$ with $q\le k \le l$ and $u_k^+\in U_w^+$, we consider the conjugation 
$$ \bx_{\xi_k}(r_k) u_k^+ \bx_{\xi_k}(-r_k)=c_ku_k^+=u_k^+d_k,$$
where $ c_k=[\bx_{\xi_k}(r_k), u_k^+]$, and $d_k=c_k\cdot [c_k^{-1},(u_k^+)^{-1}]$. We have 
$$c_k,d_k \in \prod_{a_k\ge 1, b_1,\dots, b_v\ge 0,\atop b_1+\dots +b_v\ge 1 }U_{a_k\xi_k+b_1\delta_1+\dots +b_v\delta_v},$$
by Lemma \ref{lemma38}. We write $d_k=d_k^+ d_k^-$, where $d_k^+ \in U_w^+$ and $d_k^-\in U_{w}^-$. Notice that we have 
\begin{equation}\label{eq34} a_k\xi_k+b_1\delta_1+\dots +b_v\delta_v\ne \xi_q\end{equation} by Lemma \ref{lemma37}. In fact, if $a_k\ge 2$, then we have $\rht( a_k\xi_k+b_1\delta_1+\dots +b_v\delta_v)>\rht(\xi_q)$, and thus Eq.(\ref{eq34}) is clear. If $a_k=1$, then we have $ \xi_k+b_1\delta_1+\dots +b_v\delta_v\ne \xi_q$ by Lemma \ref{lemma37} and the fact that $w(\delta_t)>0$ for each $t$. Since $ \rht( a_k\xi_k+b_1\delta_1+\dots +b_v\delta_v)> \rht(\xi_k)\ge \rht(\xi_{q+1})$, we get $d_k^-\in \prod_{t>q}U_{\xi_t}$.
Thus we get 
$$ \bx_{\xi_k}(r_k) u_k^+ \bx_{\xi_k}(-r_k)=c_ku_k^+=u_k^+d_k^+ d_k^-=u_{k+1}^+ d_k^-,$$
with $u_{k+1}^+\in U_w^+$ and $d_k^-\in \prod_{t>q}U_{\xi_t}$. If we start from $u_q^+=u^+$, repeat the above process and notice that the commutator $[U_{\xi_k}, U_{\xi_{k'}}]\subset \prod_{a,b\ge 1}U_{a\xi_k+b\xi_{k'}}$ has no nontrivial intersection with $U_{\xi_q}$ because $\rht(\xi_k)+\rht(\xi_{k'})\ge 2\rht(\xi_{q+1})>\rht(\xi_q)$, we get Claim 1.

By Claim 1, we have 
\begin{align*} &tw\bx_{\xi_l}(r_l)\dots \bx_{\xi_q}(r_q)u\\
=&tw u_1^+ \bx_{\xi_l}(r_l)\dots \bx_{\xi_q}(r_q)\bx_{\xi_{l}}(s_l)\dots \bx_{\xi_1}(s_1).
\end{align*}
Next, we switch the order of the two elements $ \bx_{\xi_l}(r_l)\dots \bx_{\xi_q}(r_q)$ and $\bx_{\xi_k}(s_k)$ for $k\ge q+1$ step by step. In each step, we get a commutator in $\prod_{a,b\ge 1}U_{a\xi_{k'}+b\xi_k}$ which has no nontrivial intersection with $U_{\xi_q}$ because $ \rht(\xi_{k'}+\xi_{k})\ge \rht(\xi_q)$. Thus finally, we get
\begin{align*} &tw\bx_{\xi_l}(r_l)\dots \bx_{\xi_q}(r_q)u\\
=&tw u_1^+ \bx_{\xi_l}(\tilde r_l)\dots \bx_{\xi_q}(\tilde r_q)\bx_{\xi_{q-1}}(s_{q-1})\dots \bx_{\xi_1}(s_1),
\end{align*}
with $\tilde r_q=r_q+s_q$. Since $\bx_{\xi_q}(r_q)\notin U_{\xi_q,m}$ but $\bx_{\xi_q}(s_q)\in U_{\xi_{q,m}}$, we get $|\tilde r_q|=|r_q|$, see Lemma \ref{lemma21}, (2). The proof of this lemma is complete if we take $\tilde u=tw u^+_1 w^{-1}t^{-1} \in U$, and $\tilde r_k=s_k$ for $k<q$.
\end{proof}

\subsection{A stability property of partial Bessel functions associated with Howe vectors}

In the following, we will fix two $\psi_U$-generic irreducible smooth representations $(\pi,V_\pi)$ and $(\pi',V_{\pi'})$ of $\Sp_{2n}(F)$ with the same central character. We fix $v\in V_{\pi}$ and $v'\in V_{\pi'}$ such that $W_v(1)=1=W_{v'}(1)$. Let $C=C(v,v')$ be a positive integer such that $v$ is fixed by $\pi(K_C)$ and $v'$ is fixed by $\pi'(K_C)$. Then we can consider the Howe vectors $v_m,v'_m$ for $m\ge C$ as defined in $\S$2. 

The main result of this section is the following
\begin{thm}\label{thm311}
Let $w\le w_0$ be a Weyl element. Let $a_t, 0\le t \le l(w)$ be a sequence of integers with $a_0=0$ and $a_t\ge t+a_{t-1}$ for all $t$ with $1\le t \le l(w)$. Let $m$ be an integer such that $m\ge 4^{a_{l(w)}}C$.
\begin{enumerate}
\item If $W_{v_k}(tw')=W_{v_k'}(tw'),$ for all $ w'<w , k\ge 4^{a_{l(w')}} C$, and $t\in T$, then
$$W_{v_m}(twu_w^-)=W_{v_m'}(twu_w^-),$$
for all $u_w^-\in U_w^--U_{w,m}^-$.
\item If $W_{v_k}(tw')=W_{v_k'}(tw'),$ for all $ w'\le w , k\ge 4^{a_{l(w')}} C$, and $t\in T$, then
$$W_{v_m}(g)=W_{v_m'}(g),$$
for all $g\in BwB$.
\end{enumerate}
\end{thm}
\noindent\textbf{Remark:} 1. We can take $a_t=t^2$ as Baruch did in \cite{Ba1}. \\
2.  Baruch proved this result for the groups $\GL_{n}, \SL_n, $, $\SO_{2n}$, $\RU(2,1)$ and $\RG\Sp_4$ for all $w\in \bW$, see Lemma 6.2.2 and Lemma 6.2.6 of \cite{Ba1}, and Proposition 5.7 (c) of \cite{Ba2}. Note that this result for $\RG\Sp_{2n}$, $\Sp_{2n}$ and $\RU(n,n)$ case are the same because these groups have the same Weyl group structure. The proof of this result in the $\RU(2,2)$ case is also given in \cite{Zh2}, which justifies some ambiguity in the proof of Lemma 6.2.6 of \cite{Ba1}. \\
3. We expect this result holds for all $w\in \bW$ for the group $\Sp_{2n}$ (without the restriction $w\le w_0$). By the previous work \cite{Zh1, Zh2} in the low rank group case, if this is true, it is possible to prove a local converse theorem for $\Sp_{2n}$ and $\RU(n,n)$.
\begin{proof}[Proof of Theorem $\ref{thm311}$]
After the preparation in $\S$3.1, in particular Proposition \ref{prop36} and Lemma \ref{lemma310}, the proof of this theorem follows from the method Baruch used to prove his Lemma 6.2.2 \cite{Ba1} directly. Since \cite{Ba1} is not published, we include a proof here.

First notice that (2) follows from (1) directly. In fact, any element $g\in BwB$ can be written as $g=u^+ tw u^-$ for $u^+\in U_w^+$ and $u^-\in U_w^-$. Thus, if we take $m\ge 4^{a_{l(w)}}C$, we have 
$$W_{v_m}(g)=\psi_U(u^+) W_{v_m}(tw u^-).$$
The same is true for $W_{v_m'}$. If $u^-\notin U_{w,m}^-$, then by (1), we get $W_{v_m}(twu^-)=W_{v_m'}(twu^-)$, and thus $W_{v_m}(g)=W_{v_m'}(g)$. If $u^-\in U_{w,m}^-$, then by Eq.(\ref{eq21}), we get $W_{v_m}(g)=\psi_U(u^+ u^-)W_{v_m}(tw)$. By assumption of (2), we get $W_{v_m}(tw)=W_{v_m'}(tw)$. Thus $W_{v_m}(g)=W_{v_m'}(g)$.

We now prove (1) by induction. If $w=1$, there is nothing to prove. For a general $w\le w_0$, we assume that (1), and hence (2) hold for all $w'$ with $w'<w\le w_0$. Let $m$ be an integer such that $m\ge 4^{l(w)}C$ by assumption. Note that the induction hypothesis and the hypothesis of (1) implies that 
$$W_{v_k}(g)=W_{v_k'}(g),$$
for all $t\in T, g\in Bw'B$, $k\ge 4^{a_{l(w')}}C$ and all $w'<w$.

 We assume $\Sigma_w^-=\wpair{\xi_1,\xi_2,\dots, \xi_l}$, where the order of the index is defined in Definition \ref{defn39}. Given $u_w^-\in U_w^--U_{w,m}^-$, we can write $u=\bx_{\xi_l}(r_l)\dots \bx_{\xi_1}(r_1)$. Let $q$ be an integer with $1\le q\le l$ such that $\bx_{\xi_t}(r_t)\in U_m$ for all $t<q$ but $\bx_{\xi_q}(r_q)\notin U_m$. Then by Lemma \ref{lemma22} or Eq.(\ref{eq21}), we have
$$W_{v_m}(twu_w^-)=\psi_U( \bx_{\xi_{q-1}}(r_{q-1})\dots \bx_{\xi_1}(r_1))W_{v_m}(tw\bx_{\xi_l}(r_l)\dots \bx_{\xi_{q}}(r_q)). $$
The same is true for $W_{v_m'}$. Thus it suffices to show that 
$$W_{v_m}(tw\bx_{\xi_l}(r_l) \dots \bx_{\xi_q}(r_q))=W_{v_m'}(tw \bx_{\xi_l}(r_l)\dots \bx_{\xi_q}(r_q)).$$
We now take an integer $k$ such that $3k\le m<4k$. By Lemma \ref{lemma22} (3), we have
$$W_{v_m}(tw\bx_{\xi_l}(r_l) \dots \bx_{\xi_q}(r_q) )=\frac{1}{\vol(U_m)}\int_{U_m} W_{v_k}(tw\bx_{\xi_l}(r_l) \dots \bx_{\xi_q}(r_q)u)\psi_U(u)^{-1}du.$$
The same is true for $W_{v_m'}$. Now (1) of the theorem follows from:

Claim 0: we have
$$ W_{v_k}(tw\bx_{\xi_l}(r_l) \dots \bx_{\xi_q}(r_q)u)=W_{v_k'}(tw\bx_{\xi_l}(r_l) \dots \bx_{\xi_q}(r_q)u), \forall u\in U_m.$$
By Lemma \ref{lemma310}, we can write
$$tw\bx_{\xi_l}(r_l)\dots \bx_{\xi_q}(r_q)u=\tilde u t w \bx_{\xi_l}(\tilde r_l) \dots \bx_{\xi_1}(\tilde r_1), $$
for some $\tilde u\in U, \tilde r_i \in F$ and $|r_q|=|\tilde r_q|$. To prove Claim 1, we consider two cases. 

Case 1, if $\bx_{\xi_t}(\tilde r_t)\in U_k$ for each $t<q$. Then by Eq.(\ref{eq22}), we have
\begin{equation}\label{eq35}W_{v_k}(tw\bx_{\xi_l}(r_l) \dots \bx_{\xi_q}(r_q)u )=\psi_U( \tilde u \bx_{\xi_{q-1}}(\tilde r_{q-1})\dots \bx_{\xi_1}(\tilde r_1))W_{v_k}(tw \bx_{\xi_l}(\tilde r_l)\dots \bx_{\xi_q}(\tilde r_q) ).\end{equation}
By assumption, we have $\bx_{\xi_q}(r_q)\notin U_m$, and thus $ r_q\notin \CP^{-(2\rht(\xi_q)-1)m}$ by Lemma \ref{lemma21}. Since $|\tilde r_q|=|r_q|$, we get $\tilde r_q \notin \CP^{-(2\rht(\xi_q)-1)m}$. Thus $ \tilde r_q^{-1}\in \CP^{(2\rht(\xi_q)-1)m}\subset \CP^{(2\rht(\xi_q)+1)k}$, since $3k\le m$. Thus by Lemma \ref{lemma21}, we get 
$$\bx_{-\xi_q}(-\tilde r_q^{-1})\in U_{-\xi_q,k}\subset H_k.$$
By Lemma \ref{lemma22} or Eq.(\ref{eq21}), we get 
\begin{equation}\label{eq36} W_{v_k}(tw \bx_{\xi_l}(\tilde r_l)\dots \bx_{\xi_q}(\tilde r_q) )=W_{v_k}(tw \bx_{\xi_l}(\tilde r_l)\dots \bx_{\xi_q}(\tilde r_q)\bx_{\xi_q}(-\tilde r_q^{-1}) ). \end{equation}
By Proposition \ref{prop36}, we get
$$ tw \bx_{\xi_l}(\tilde r_l)\dots \bx_{\xi_q}(\tilde r_q)\bx_{\xi_q}(-\tilde r_q^{-1})\in Bw'B,$$
for some $w'<w$. Notice that $k>\frac{1}{4}m\ge 4^{a_{l(w)}-1}C\ge 4^{a_{l(w')}}C$. Thus by the induction hypothesis and the hypothesis of (1), we get 
$$W_{v_k}( tw \bx_{\xi_l}(\tilde r_l)\dots \bx_{\xi_q}(\tilde r_q)\bx_{\xi_q}(-\tilde r_q^{-1}))=W_{v_k'}(tw \bx_{\xi_l}(\tilde r_l)\dots \bx_{\xi_q}(\tilde r_q)\bx_{\xi_q}(-\tilde r_q^{-1})).$$
By Eq.(\ref{eq35}, \ref{eq36}) and their corresponding parts for $W_{v_k'}$, we get Claim 0 in Case 1.

Case 2, it's not true that $\bx_{\xi_t}(\tilde r_t)\in U_k$ for each $t<q$. We then take a $q_1$ such that $ \bx_{\xi_t}(\tilde r_t)\in U_k$ for all $t<q_1$ but $\bx_{\xi_{q_1}}(\tilde r_{q_1})\notin U_k$. Note that $q_1<q$ by assumption. We then take an integer $k_1$ such that $3k_1\le k<4k_1$, and write 
$$W_{v_k}(tw\bx_{\xi_l}(\tilde r_l) \dots \bx_{\xi_{q_1}}(\tilde r_q) )=\frac{1}{\vol(U_k)}\int_{U_k} W_{v_{k_1}}(tw\bx_{\xi_l}(\tilde r_l) \dots \bx_{\xi_q}(\tilde r_{q_1})u)\psi_U(u)^{-1}du.$$

Then we make the following

Claim 1: We have
$$W_{v_{k_1}}(tw\bx_{\xi_l}(\tilde r_l) \dots \bx_{\xi_q}(\tilde r_{q_1})u)=W_{v_{k_1}}(tw\bx_{\xi_l}(\tilde r_l) \dots \bx_{\xi_q}(\tilde r_{q_1})u), \forall u\in U_{k_1}. $$

Note that Claim 1 implies Claim 0. To prove Claim 1, we repeat the above process. The process will terminate after $q\le l=l(w)$ steps. Note that in the $t$-th step, we need to take an integer $k_t$ with $k_t>\frac{1}{4}k_{t-1}$. Thus
$$k_t\ge \frac{1}{4^{l(w)}}m\ge 4^{a_{l(w)}-l(w)}C\ge 4^{a_{l(w')}}C,$$
for each $w'<w$. Thus the induction hypothesis applies in each step. This completes the proof.
\end{proof}

Before we state the consequences of Theorem \ref{thm311}, we need the following
\begin{lem}\label{lemma312}
\begin{enumerate}
\item If the residue field of $F$ has odd characteristic, then $W_{v_m}(g)=W_{v_m'}(g)$ for all $g\in B, $ and $m\ge C$.
\item We have $W_{v_m}(tw)=0=W_{v_m'}(tw)$ for all $t\in T$, $m\ge C$, $w<w_0$ and $w\ne 1$.
\end{enumerate}
\end{lem}
\begin{proof}
(1) This follows from Corollary \ref{cor25} and the fact that $\pi$ and $\pi'$ have the same central character.\\
(2) We claim that for all $w<w_0$ and $w\ne 1$, there exists a simple root $\gamma$ such that $w(\gamma)$ is positive but not simple. We first show that this claim implies $W_{v_m}(tw)=0=W_{v_m'}(tw)$ for all $m\ge C,t\in T$, $w\le w_0$ and $w\ne 1$. In fact, suppose that $\gamma$ is a simple root but $w(\gamma)$ is a positive but non-simple root. Take $r\in \CP^{-m}$, we have $\bx_{\gamma}(r)\in U_m$. From the relation
$$tw \bx_{\gamma}(r)=\bx_{w(\gamma)}(\gamma(t)r) tw,$$
and Eq.(\ref{eq21}), we get
$$ \psi_m(\bx_{\gamma(r)})W_{v_m}(tw)=\psi_U(\bx_{w(\gamma)}(\gamma(t)r))W_{v_m}(tw).$$
Since $w(\gamma)$ is not simple, we get $\psi_U(\bx_{w(\gamma)}(\gamma(t)r))=1 $. It is clear that $ \psi_m(\bx_{\gamma(r)})=\psi(r)$. Then we get $(\psi(r)-1)W_{v_m}(tw)=0$. Since $\psi$ is a nontrivial additive character with conductor $\CO$, we can choose $r\in \CP^{-m}$ such that $\psi(r)\ne 1$. Thus $W_{v_m}(tw)=0$. The same argument shows that $W_{v_m'}(tw)=0$.

Next, we prove the claim. By Proposition 3.2 of \cite{CPS}, it suffices to show that $w_lw_0$ is the long Weyl element of the Levi subgroup $M_{w_0}$ of a maximal parabolic subgroup $P_{w_0}\supset B$, where $$w_l=\begin{pmatrix} &J_n\\ -J_n& \end{pmatrix}$$ is the long Weyl element of $\Sp_{2n}$. We have 
$$w_lw_0=\begin{pmatrix}-1& &&\\ &&J_{n-1}&\\ &-J_{n-1}&&\\ &&&-1 \end{pmatrix},$$
which is the long Weyl element of $M_{w_0}\cong \GL_1\times \Sp_{2n-2}$. It is clear that the corresponding parabolic subgroup $P_{w_0}$ is a maximal parabolic subgroup. This proves the claim and hence the lemma.
\end{proof}

\begin{cor}\label{cor313} Suppose that the field $F$ has odd residue characteristic.
\begin{enumerate}
\item Given $w\in \bW$ with $w<w_0$, and $m\ge 4^{l(w)^2}C$, we have 
$$W_{v_m}(g)=W_{v_m'}(g),$$
for all $g\in BwB$.
\item For $m\ge 4^{l(w_0)^2}C$ and $u\in U_{w_0}^--U_{w_0,m}^-$, we have 
$$W_{v_m}(tw_0u)=W_{v_m'}(tw_0u),$$
for all $t\in T$.
\end{enumerate}
\end{cor}
\begin{proof}
This is a direct consequence of Theorem \ref{thm311} and Lemma \ref{lemma312}
\end{proof}
We remark that Corollary \ref{cor313} is the key to prove our main theorem, see the proof of Theorem \ref{thm44}.

\section{Stability of $\gamma$-factors}
\subsection{Howe vectors for the Weil representations of $\widetilde \SL_2(F)$}
Given an unramified additive character $\psi$ of $F$, recall that we have a Weil representation $\omega_{\psi^{-1}}$ of $\widetilde \SL_2(F)$ on $\CS(F)$. For an integer $m>0$, let $\phi_m\in \CS(F)$ be the characteristic function of $\CP^{(2n-1)m}$, which will play the role of Howe vectors for the Weil representations. 

\begin{lem}\label{lemma41}
Suppose the residue characteristic of $F$ is odd. We have $$\omega_{\psi^{-1}}(\bn_1(b))\phi^m=\phi^m, \textrm{ for } b\in \CP^{-(4n-3)m}$$ and 
$$\omega_{\psi^{-1}}(\bar \bn_1(b))\phi^m=\phi^m, \textrm{ for } b\in \CP^{(4n-1)m}.$$ Here by abuse notation, we do not distinguish an element $g\in \SL_2(F)$ with $(g,1)\in \widetilde \SL_2(F)$.
\end{lem}
\begin{proof}
 For $x\in F$, we have
$$\omega_{\psi^{-1}}(\bn_1(b))\phi_m(x)=\psi(bx^2)\phi_m(x).$$
For $x\in \Supp(\phi^m)=\CP^{(2n-1)m}$ and $b\in \CP^{-(4n-3)m}$, we have $bx^2\in \CP^m\subset \CO$, and thus $\psi(bx^2)=1$. Now it is clear that $\omega_{\psi^{-1}}(u)\phi_m=\phi_m$. 

To prove the second formula, we write $\bar \bn_1(b)=(w^1)^{-1} \bn_1(-b)w^1 $, with $b\in \CP^{(4n-1)m}$. Denote $\phi_m'=\omega_{\psi^{-1}}(w^1)\phi_m$. We have
\begin{align*}
\phi_m'(x)&=\omega_{\psi^{-1}}(w^1)\phi_m(x)\\
&=\gamma(\psi^{-1})\int_F \phi_m(y)\psi^{-1}(2xy)dy\\
&=\gamma(\psi^{-1}) \int_{\CP^{(2n-1)m}}\psi^{-1}(2xy)dy\\
&=\gamma(\psi^{-1})q_F^{-(2n-1)m}\textrm{Char}_{\CP^{-(2n-1)m}}(x),
\end{align*}
where $\textrm{Char}_{\CP^{-(2n-1)m}}$ is the characteristic function of $\CP^{-(2n-1)m}$. In the above A similar argument as above shows that $\omega_{\psi^{-1}}(\bn_1(-b))\phi_m'=\phi_m'$ for $b\in \CP^{(4n-1)m}$. Thus we get
$$ \omega_{\psi^{-1}}(\bar \bn_1(b))\phi_m'= \omega_{\psi^{-1}}((w^1)^{-1}) \omega_{\bn_1(-b)}\phi_m'= \omega_{\psi^{-1}}((w^1)^{-1})\phi_m'=\phi_m.$$
This finishes the proof of the lemma.
\end{proof}

\subsection{Sections of genuine induced representations of $\widetilde \SL_2(F)$} In this subsection, we construct some sections of genuine induced representations of $\widetilde \SL_2(F)$, which will be used in the proof of stability of $\gamma$-factors for $\Sp_{2n}$. The same constructions has been used in \cite{ChZh} to get a local converse theorem for $\SL_2$.

Note that $\bar U^1(F)$ and $U^1(F)$ splits in $\Mp_2(F)$.
Moreover, for $g_1\in U^1$ and $g\in \bar U^1$ we have $c(g_1,g_2)=1$.
In fact, if $g_1=\bn_1(y)$ and $g_2=\bar \bn(x)$ with $x\ne 0$, we have
$\bx(g_1)=1$ and $\bx(g_2)=x$, and thus
$$c(g_1,g_2)=( 1,x)_F(-x, x)_F=1.$$
This shows that $\bar U^1 \cdot  U^1\subset \SL_2(F),$ where
$\SL_2(F)$ denotes the subset of $\widetilde \SL_2(F)$ which
consists elements of the form $(g,1)$ for $g\in \SL_2(F)$.

For a positive integer $i$, we denote $$U^1_i=\begin{pmatrix}1& \CP^{-i} \\ &1\end{pmatrix}, \textrm{ and } \bar U^1_i=\begin{pmatrix}1& \\ \CP^{3i}&1 \end{pmatrix} .$$
Note that $U^1_i=U^1\cap H_i$ and $\bar U^1_i= \bar U^1\cap H_i$, where we view $U^1_i$ and $\bar U^1_i$ as a subgroup of $\Sp_{2n}$ by the standard embedding $\SL_2\incl \Sp_{2n}$.

Let $X$ be an open compact subgroup of $U^1$. For $x\in X$ and
$i>0$, we consider the set $A(x,i)=\wpair{\bar u\in \bar U^1: \bar u x\in B^1 \cdot \bar U^1_{i}}$, where $B^1$ is the upper triangular Borel of $\SL_2$. Note that the definition of $A(x,i)$ makes sense because $ \bar U^1\cdot U^1\subset \SL_2$, as we showed above.
\begin{lem}\label{lemma42}
\begin{enumerate}
\item For any positive integer $c$, there exists an integer $i_1=i_1(X,c)$ such that for all
$i\ge i_1$, $x\in X$ and $ \bar u\in A(x,i)$, we have
$$\bar u x=u \bm_1(a) \bar u_0$$
with $u\in U^1, \bar u_0 \in \bar  U^1_i$ and $a\in 1+\CP^c$.
\item There exists an integer $i_0=i_0(X)$ such that for all $i\ge i_0$, we have $A(x,i)=\bar U^1_{i}$
for all $i\ge i_1$.
\end{enumerate}
\end{lem}
\begin{proof}
 Since $X$ is compact, there is a constant $C$ such that $|x|<C$ for all $\bn_1(x)\in X\subset N$.

For $\bn_1(x)\in X, \bar \bn_1(y)\in A(\bn_1(x),i)$, we have $\bar \bn_1(y)\bar n_1(x) \in B^1 \cdot \bar U^1_{i}$, thus we can
assume that $$\bar \bn_1(y) \bn_1(x)=\begin{pmatrix}a& b\\ & a^{-1} \end{pmatrix} \bar \bn_1(\bar y)$$ for
$a\in F^\times, b\in F$ and $\bar y \in \CP^{3i}$. Rewrite the above expression as
$$\bar \bn_1 (-y) \begin{pmatrix} a& b \\ & a^{-1} \end{pmatrix}= \bn_1(x)\bar \bn_1(-\bar y),$$
or

$$\begin{pmatrix}a & b\\ -ay & a^{-1}-by \end{pmatrix}=\begin{pmatrix}1-x\bar y & x\\ -\bar y &1 \end{pmatrix}.$$
Thus we get $$a=1-x\bar y, ay=\bar y.$$
Since $|x|<C$ and $\bar y\in \CP^{3i}$, it is clear that for any positive integer $c$, we can
choose $i_1(X,c)$ such that $a=1-x\bar y\in 1+\CP^c$ for all $ \bn_1(x)\in X$ and $\bar \bn(y)\in A(\bn_1(x),i)$. This proves (1).

If we take $i_0(X)=i_1(X,1)$, we get $a\in 1+\CP\subset \CO^\times$ for $i\ge i_0$. From $ay=\bar y$,
we get $y\in \CP^{3i}$. Thus we get that for $i\ge i_0(X)$, we have $ \bar \bn_1(y)\in \bar U^1_{i}$, i.e.,
$A(x,i)\subset \bar U^1_{i}$.
The other direction inclusion can be checked similarly if $i$ is large. We omit the details.
\end{proof}

Now let $\eta$ be a quasi-character of $F^\times$. Given a positive integer $i$ and a complex number $s\in \BC$, we
consider the following function $f^i_s$ on $\widetilde \SL_2(F)$:
$$f_s^i((g,\zeta))=\left\{\begin{array}{lll}\zeta \mu_{\psi^{-1}}(a)^{-1}\eta_{s+1/2}(a),& \textrm{ if } g=
\left(\begin{pmatrix}a & b\\ & a^{-1}\end{pmatrix}, \zeta \right)  \bar \bn_1(x), \textrm{ with }\\
& \quad a\in F^\times, b\in F, \zeta\in \mu_2, x\in \CP^{3i}, \\ 0, &\textrm{ otherwise.} \end{array}\right.$$

\begin{lem}\label{lemma43} Suppose that the residue characteristic of $F$ is odd.
\begin{enumerate}
\item There exists an integer $i_2(\eta)$ such that for all $i\ge i_2(\eta)$, $f_s^i$ defines a section
in $\tilde I(s,\eta,\psi^{-1})$.
\item Let $X$ be an open compact subset of $U^1$, then there exists an integer $I(X,\eta)\ge i_2(\eta)$
such that for all $i\ge I(X,\eta)$, we have
$$\tilde f_s^i(w^1 x)=\vol(\bar U^1_{i})=q_F^{-3i},$$
for all $x\in X$, where $\tilde f_s^i=M_s(f_s^i)$.
\end{enumerate}
\end{lem}
Recall that $w=\begin{pmatrix}&1\\ -1& \end{pmatrix}$.
\begin{proof}
(1) From the definition, it is clear that $$f_s^i\left(
\left(\begin{pmatrix} a&b\\ & a^{-1}\end{pmatrix}, \zeta\right)
\tilde g \right)=\zeta\mu_{\psi^{-1}}(a)^{-1}\eta_{s+1/2}(a)f_s^i(\tilde
g),$$ for $a\in F^\times, b\in F, \zeta \in \mu_2, $ and $\tilde
g\in \widetilde \SL_2(F)$. It suffices to show that for $i$ large,
there is an open compact subgroup $\widetilde H_i\subset \widetilde
\SL_2(F)$ such that $f_s^i(\tilde g \tilde h)=f_s^i (\tilde g)$ for
all $\tilde g\in \widetilde \SL_2(F),$ and $\tilde h\in \widetilde
H_i$.

If $\psi$ is unramified and the residue characteristic is not 2 as we assumed, the character
$\mu_{\psi^{-1}}$ is trivial on $\CO_F^\times$, see \cite{Sz2} for example.

Let $c$ be a positive integer such that $ \eta$ is trivial on $1+\CP^c$. Denote $K_c^1=1+\Mat_{2\times 2}(\CP^c)$, the standard congruence subgroup of $\SL_2(F)$. Let $i_2(\eta)=\wpair{c, i_0(U^1\cap K^1_c ), i_1(U^1\cap K^1_c, c) }.$ For $i\ge i_2(\eta)$, we take
$\widetilde H_i= K^1_{4i}=1+M_2(\CP^{4i})$. Note that the double cover map $\widetilde \SL_2\ra \SL_2$ splits over $K^1_{4i}$, and thus we can view $K_{4i}^1$
as a subgroup of $\widetilde \SL_2$. We now check that for $i\ge i_2(\eta)$, we have
$f_s^i(\tilde g h)=f_s(\tilde g)$ for all $\tilde g\in \widetilde \SL_2$ and $h\in K_{4i}$.
We have the Iwahori decomposition $K^1_{4i}=(U^1\cap K_{4i})(A\cap K_{4i})(\bar U^1\cap K_{4i})$.

For $h\in \bar U^1\cap K^1_{4i}\subset \bar U^1_i$, it is clear that $f_s^i(\tilde g h)=f_s^i(\tilde g)$
by the definition of $f_s^i$. 

Now we take $h\in A\cap K^1_{4i}$. Write $h=\bm_1(a_0)$, with $a_0\in 1+\CP^{4i}$.
We have $\bar \bn_1(x) h= h \bar \bn_1(a_0^{-2} x)$. It is clear that $x\in \CP^{3i}$ if and only if $a_0^{-2}x\in \CP^{3i}$.
On the other hand, for any $a\in F^\times, b\in F$, we have
$$c\left( \begin{pmatrix} a& b\\ & a^{-1} \end{pmatrix}, \bm_1(a_0) \right)=(a^{-1},a_0^{-1})=1,$$
since $a_0\in 1+\CP_F^{4i}\subset F^{\times, 2}$ by Lemma \ref{lemma24}. Thus
we get
$$ \left( \begin{pmatrix} a& b\\ & a^{-1} \end{pmatrix}, \zeta \right)\bar \bn_1(x)h=
\left( \begin{pmatrix} aa_0& ba_0^{-1}\\ & a^{-1}a_0^{-1} \end{pmatrix}, \zeta \right)\bar \bn_1(a_0^{-1}x). $$
By the definition of $f_s^i$, if $x\in \CP^{3i}$, for $g=\left(\begin{pmatrix} a& b\\ & a^{-1} \end{pmatrix},
\zeta \right)\bar \bn_1(x)$ we get
$$f_s^i(g h)=\zeta\mu_{\psi^{-1}}(aa_0)^{-1}\eta_{s+1/2}(aa_0)=\zeta \mu_{\psi^{-1}}(a)^{-1}\eta_{s+1/2}(a)=f_s^i(g),$$
by the assumption on $i$.

Finally, we consider $h\in  U^1\cap K^1_{4i}\subset  U^1\cap K_c$. By assumption on $i$, we get
$$A(h,i)=A(h^{-1},i)=\bar U^1_i.$$
In particular, for $\bar u\in \bar U^1_i$, we have $ \bar u h \in B^1\cdot \bar U^1_i$ and
$\bar u h^{-1}\in B^1\cdot \bar U^1_i$. Now it is clear that $\tilde g\in \widetilde B^1\cdot \bar U^1_i$
if and only if $\tilde g h\in \widetilde B^1\cdot \bar U^1_i$. Thus $f_s^i(\tilde g)=0$ if and only if
$f_s^i(\tilde gh)=0$. Moreover, for $\bar u \in \bar U^1_i$, we have
$$\bar u h=\begin{pmatrix}a_0 & b_0\\ & a_0^{-1} \end{pmatrix}\bar u_0,$$
for $a_0\in 1+\CP^c$, $b_0\in F$ and $\bar u_0\in \bar U^1_i$. Thus for $\tilde g=
\left(\begin{pmatrix}a&b \\ & a^{-1}  \end{pmatrix}, \zeta\right) \bar u$ with $\bar u \in \bar U^1_i$,
we get
$$\tilde gh =\left( \begin{pmatrix}aa_0 & ab_0+a_0^{-1}b \\ & a_0^{-1}a^{-1} \end{pmatrix},\zeta \right)
\bar u_0.$$
Here we used the fact that $a_0\in 1+\CP^c$ is a square, and thus
$$c\left(\begin{pmatrix}a& b\\ & a^{-1} \end{pmatrix},\begin{pmatrix}a_0 & b_0\\ & a_0^{-1} \end{pmatrix}  \right)=1.$$

Since $\mu_{\psi^{-1}}(a_0)=1$, $(a,a_0)=1$ and $\eta_{s+1/2}(a_0)=1$, we get
$$f_s^i(\tilde g h)=f_s^i(g).$$
This finishes the proof of (1).

(2)   As in the proof of (1), let $c$ be a positive integer such that $\eta$ is trivial on $1+\CP^c$.
Take $I(X,\eta)=\max\wpair{i_1(X,c), i_0(X)}$. We have $$\tilde f_s^i(w^1 x)=\int_{F} f_s^i(((w^1)^{-1} \bn_1(b),1) w^1 x)db=\int_{F}f_s^i( (w^1)^{-1}\bn_1(b)w^1x,1)db.$$
By the definition of $f_s^i$, $f_s^i((w^{1})^{-1} \bn_1(b) w^1 x)\ne 0$ if and only if $(w^1)^{-1} \bn_1(b) w^1 x\in B^1\bar U^1_{i}$,
if and only if $(w^1)^{-1} \bn_1(b) w^1 \in A(x,i)=\bar U^1_{i}$ for all $i\ge I(X),$ and $ x\in X$.
On the other hand, if $(w^1)^{-1} \bn_1(b) w^1\in A(x,i)$, we have
$$(w^1)^{-1} \bn_1(b) w^1 x= \begin{pmatrix} a & b_1\\ & a^{-1}\end{pmatrix}\bar u_0, $$
with $a\in 1+\CP^c$ by Lemma \ref{lemma42}. Thus
$$f_s^i( (w^1)^{-1} \bn_1(b) w^1x)=\eta_{s+1/2}(a)\mu_{\psi^{-1}}(a)=1.$$ Now it is clear that
$$\tilde f_s^i(w^1x)=\vol(\bar U^1_i) =q_F^{-3i}.$$
\end{proof}

\subsection{The stability of gamma factors for generic representations of $\Sp_{2n}(F)$ when the characteristic of $F$ is odd}
Recall the notations from $\S$3.2. We assume that $F$ is a $p$-adic field with odd residue characteristic, $(\pi,V_\pi)$ and $(\pi',V_{\pi'})$ are two irreducible smooth $\psi_U$-generic representations of $\Sp_{2n}(F)$ with the same central character. We take $v\in V_{\pi}, v'\in V_{\pi'}$ such that $W_{v}(1)=1=W_{v'_1}$, and let $C=C(v,v')$ be an integer such that $v$ and $v'$ are fixed by $K_C$ under the action of $\pi$ and $\pi'$ respectively. Now we can state the main theorem of the paper:
\begin{thm}\label{thm44}
There is an integer $l=l(\pi,\pi')$ such that for any quasi-character $\eta$ of $F^\times$ with $\cond(\eta)>l$, we have
$$\gamma(s,\pi,\eta,\psi)=\gamma(s,\pi',\eta,\psi).$$
\end{thm}

\begin{proof}
 We take a quasi-character $\eta$ of $F^\times$ with conductor $\cond(\eta)$. Let $m$ be an integer such that $m\ge \max\wpair{\cond(\eta), 4^{l(w_0)^2}C}$, and $i$ be an integer such that $i\ge \max\wpair{ i_2(\eta), I( \widetilde U^1_m, \eta), (4n-1)m/3 }$, where $\widetilde U^1_m=\wpair{\bn_1(x),x\in \CP^{-(4n-3)m}}$,  and $I( \widetilde U^1_m, \eta) $ is defined in Lemma \ref{lemma43}. By Lemma \ref{lemma43}, we have a section $f_s^i\in \tilde I(s,\eta,\psi^{-1})$.

Let $W_m$ be $W_{v_m}$ or $W_{v_m'}$. We compute $\Psi(W_m,\phi_m, f_s^i)$. We take the integral over $U^1\setminus \SL_2$ on the open dense subset $U^1\setminus U^1A \bar U^1\subset U^1\setminus \SL_2$
\begin{align*}
\Psi(W_m,\phi_m, f_s^i)&=\int_{U^1\setminus \SL_2}\int_{F^{n-2}}\int_FW_m(j(r(y,x)g)) \omega_{\psi^{-1}}(g)\phi_m(x)f_s^i(g)dx dy dg\\
&=\int_{F^\times \times F}\int_{F^{n-2}}\int_F W_m(j(r(y,x)) j(\bm_1(a)) j(\bar \bn_1(b))) \\
&\cdot \omega_{\psi^{-1}}( \bm_1(a) \bar \bn_1(b))\phi_m(x) f_s^i( \bm_1(a) \bar \bn_1(b))dxdy db |a|^{-2} da.
\end{align*}
By definition of $f_s^i$, we get
\begin{align*}
\Psi(W_m,\phi_m, f_s^i)&=\int_{F^\times \times \CP^{3i}}\int_{F^{n-2}}\int_F W_m(j(r(y,x)) j(\bm_1(a)) j(\bar \bn_1(b))) \\
&\cdot \omega_{\psi^{-1}}( \bm_1(a) \bar \bn_1(b))\phi_m(x)(\mu_{\psi^{-1}}(a))^{-1} \eta_{s-3/2}(a)dxdy db da
\end{align*}
By assumption, we have $3i> (4n-1)m$ so that $\CP^{3i}\subset \CP^{(4n-1)m}$. Thus for $b\in \CP^{3i}\subset \CP^{(4n-1)m}$, we get $j(\bar \bn_1(b))=\bx_{-2(\alpha_1+\dots \alpha_{n-1})-\beta}(b)\in H_m$. By Lemma \ref{lemma22} and the fact $\omega_{\psi^{-1}}(\bar \bn_1(b))\phi_m=\phi_m$ (Lemma \ref{lemma41}), we get
\begin{align*}
\Psi(W_m,\phi_m, f_s^i)&=q_F^{-3i}\int_{F^\times }\int_{F^{n-2}}\int_F W_m(j(r(y,x)) j(\bm_1(a)) ) \\
&\cdot \omega_{\psi^{-1}}( \bm_1(a) )\phi_m(x) \mu_{\psi^{-1}}(a)^{-1}\eta_{s-3/2}(a)dxdy  da.
\end{align*}
Since $r(y,x)\bm_1(a)=\bm_1(a)r(ya,xa)$, and $\omega_{\psi^{-1}}(\bm_1(a))\phi_m(x)=\mu_{\psi^{-1}}(a)|a|^{1/2} \phi_m(xa)$, by changing variables, we get

\begin{align*}
\Psi(W_m,\phi_m, f_s^i)&=q_F^{-3i}\int_{F^\times }\int_{F^{n-2}}\int_F W_m( j(\bm_1(a)j(r(y,x))) ) \phi_m(x) \eta_{s-n}(a)dxdy  da\\
&=q_F^{-3i} \int_{F^\times }\int_{F^{n-2}}\int_{\CP^{(2n-1)m}} W_m( \bt(a) j(r(y,x))) )\eta_{s-n}(a)dxdy  da,
\end{align*}
where in the last step, we used the definition of $\phi_m=\textrm{char}_{\CP^{(2n-1)m}}.$ Since for $x\in \CP^{(2n-1)m}$, we have $j(r(0,x))=\bx_{-(\alpha_1+\dots+\alpha_{n-1})}(x)\in H_m$ (see Lemma \ref{lemma21}), we get
\begin{align*}
\Psi(W_m,\phi_m, f_s^i)&=q_F^{-3i}\int_{F^\times }\int_{F^{n-2}}\int_F W_m( j(\bm_1(a)j(r(y,x))) ) \phi_m(x) \eta_{s-n}(a)dxdy  da\\
&=q_F^{-3i-(2n-1)m} \int_{F^\times }\int_{F^{n-2}}W_m( \bt(a) j(r(y,0))) )\eta_{s-n}(a)dy  da.
\end{align*}
By Lemma \ref{lemma26}, we get 
\begin{align*}
\Psi(W_m,\phi_m, f_s^i)&=q_F^{-3i-(2n-1)m} q_F^{-\sum_{i=1}^{n-2}(2i+1)m}\int_{F^\times }W_m( \bt(a))\eta_{s-n}(a) da\\
&=q_F^{-3i-(n^2-1)m}\int_{F^\times} W_m(\bt(a))\eta_{s-n}(a)da.
\end{align*}
Finally, by Corollay \ref{cor25}, we get $W_m(\bt(a))=0$ if $a\notin 1+\CP^m$ and $ W_m(\bt(a))=1$ if $a\in 1+\CP^m$. Notice that $m\ge \cond(\eta)$ by assumption, we get
$$\Psi(W_m,\phi_m,f_s^i)=q_F^{-3i-(n^2-1)m}\int_{1+\CP^m}\eta_{s-n}(a)da=q_F^{-3i-(n^2-1)m}\vol(1+\CP^m)=q_F^{-3i-n^2m}.$$
Note that this calculation works form both $W_{v_m}$ and $W_{v_m'}$, we then get
\begin{equation}\label{eq41}
\Psi(W_{v_m},\phi_m,f_s^i)=\Psi(W_{v_m'}, \phi_m, f_s^i)=q_F^{-3i-n^2m}.
\end{equation}

Next, we compute the other side of the functional equation, i.e., $\Psi(W_m,\phi_m, \tilde f_s^i)$. We replace the domain $U^1\setminus \SL_2$ by its open dense subset $U^1\setminus U^1Aw^1 U^1$. Thus
\begin{align*}
\Psi(W_m, \phi_m, \tilde f_s^i)&=\int_{F^\times \times F} \int_{F^{n-2}}\int_F W_m(j(r(y,x) \bm_1(a)w^1 \bn_1(b))) \omega_{\psi^{-1}}(\bm_1(a)w^1\bn_1(b))\phi_m(x)\\
&\cdot \tilde f_s^i(\bm_1(a)w^1 \bn_1(b))dx dy db |a|^{-2}da.
\end{align*}
Notice that $j(w^1)=w_0$ and 
\begin{align*}
j(r(y,x) \bm_1(a)w^1 \bn_1(b))&=j(\bm_1(a) r(ya,xa) w^1 \bn_1(b))\\
&=j(\bm_1(a) w^1 r'(ya,xa) \bn_1(b))\\
&=\bt(a) w_0 j(r'(ya,xa))j(\bn_1(b))
\end{align*}
where $r'(y,x)=s_\beta^{-1} r(y,x) s_\beta$. By changing variables, we get
\begin{align*}
\Psi(W_m, \phi_m, \tilde f_s^i)&=\int_{F^\times \times F} \int_{F^{n-2}}\int_F W_m( \bt(a) w_0 j(r'(y,x))j(\bn_1(b)))\\
&\cdot \mu_{\psi^{-1}}(a)|a|^{1/2}\omega_{\psi^{-1}}(w^1\bn_1(b))\phi_m(x)  \tilde f_s^i(\bm_1(a)w^1 \bn_1(b))dx dy db |a|^{-1-n}da.
\end{align*}
We then get

\begin{align*}
&\Psi(W_{v_m},\phi_m, \tilde f_s^i)-\Psi(W_{v_m'},\phi_m,\tilde f_s^i)\\
=& \int_{F^\times \times F} \int_{F^{n-2}}\int_F \left( W_{v_m}( \bt(a) w_0 j(r'(y,x))j(\bn_1(b))) -W_{v_m'}( \bt(a) w_0 j(r'(y,x))j(\bn_1(b)))  \right)\\
&\cdot \mu_{\psi^{-1}}(a)|a|^{1/2}\omega_{\psi^{-1}}(w^1\bn_1(b))\phi_m(x)  \tilde f_s^i(\bm_1(a)w^1 \bn_1(b))dx dy db |a|^{-1-n}da.
\end{align*}

In matrix form, we have
$$r'(y,x)=\bn_n \begin{pmatrix}y&&\\ x&&\\ 0 &x& y \end{pmatrix},$$
and
$$j(r'(y,x))j(\bn_1(b))=\bn_n\begin{pmatrix}x& {}^t\! y J_{n-2} &b\\ &&y \\ &&x \end{pmatrix}.$$
We have $j(r'(y,x))j(\bn_1(b))\in U_{w_0}^-$. By Corollary \ref{cor313}, if $ j(r'(y,x))j(\bn_1(b)) \notin U_{m}$, we have 
$$W_{v_m}( \bt(a) w_0 j(r'(y,x))j(\bn_1(b))) -W_{v_m'}( \bt(a) w_0 j(r'(y,x))j(\bn_1(b)))=0, $$
and thus 
\begin{align*}
&\Psi(W_{v_m},\phi_m, \tilde f_s^i)-\Psi(W_{v_m'},\phi_m,\tilde f_s^i)\\
=& \int_{F^\times} \int_{D_m} \left( W_{v_m}( \bt(a) w_0 j(r'(y,x))j(\bn_1(b))) -W_{v_m'}( \bt(a) w_0 j(r'(y,x))j(\bn_1(b)))  \right)\\
&\cdot \mu_{\psi^{-1}}(a)|a|^{1/2}\omega_{\psi^{-1}}(w^1\bn_1(b))\phi_m(x)  \tilde f_s^i(\bm_1(a)w^1 \bn_1(b))dx dy db |a|^{-1-n}da,
\end{align*}
where $D_m=D\cap U_m$ with $$D=\wpair{ j(r'(y,x))j(\bn_1(b))=\bn_n \begin{pmatrix}x& {}^t\! y J_{n-2} &b\\ &&y \\ &&x   \end{pmatrix}, x, b\in F, y\in \Mat_{(n-2)\times 1}(F)}.$$
Now suppose that $j(r'(y,x))j(\bn_1(b))\in D_m \subset U_m$, then by Eq.(\ref{eq21}), we have
$$ W_{v_m}( \bt(a) w_0 j(r'(y,x))j(\bn_1(b)))=W_{v_m}(\bt(a)w_0), W_{v_m'}( \bt(a) w_0 j(r'(y,x))j(\bn_1(b)))=W_{v_m'}(\bt(a)w_0).$$
For $ j(r'(y,x))j(\bn_1(b))\in D_m \subset U_m$, we have $b\in \CP^{-(4n-3)m}$ and $x\in \CP^{-(2n-1)m}$, and thus we get 
$$ \omega_{\psi^{-1}}( w^1\bn_1(b))\phi_m(x)= \omega_{\psi^{-1}}(w^1)\phi_m(x)=\gamma(\psi^{-1})q_F^{-(2n-1)m},$$
see Lemma \ref{lemma41} and its proof. On the other hand, we have
$$ \tilde f_s^i(\bm_1(a)w^1 \bn_1(b))=(\mu_{\psi^{-1}}(a))^{-1}\eta^{-1}_{-s+3/2}(a) q_F^{-3i},$$
Lemma \ref{lemma43} and the assumption that $i>I(\widetilde U^1_m, \eta)$.

From the above discussions, we get
\begin{align}
&\Psi(W_{v_m},\phi_m, \tilde f_s^i)-\Psi(W_{v_m'},\phi_m,\tilde f_s^i)\label{eq42} \\
=& \gamma(\psi^{-1})\vol(D_m) q_F^{-3i-(2n-1)m} \int_{F^\times}  \left( W_{v_m}( \bt(a) w_0 ) -W_{v_m'}( \bt(a) w_0)  \right) \eta^{-1}_{-s-n+1}(a)da .\nonumber
\end{align}

Let $k=4^{l(w_0)^2}C$. By Lemma \ref{lemma32}, we have
\begin{align*}
&W_{v_m}(\bt(a)w_0)-W_{v_m'}(\bt(a)w_0)\\
=&\frac{1}{\vol(U_m)} \int_{U_m} (W_{v_k}(\bt(a)w_0 u)-W_{v_k'}(\bt(a)w_0u))\psi_U^{-1}(u)du\\
=&\frac{1}{\vol(U_m)}\int_{U_{w_0,m}^-}\int_{U_{w_0,m}^+} (W_{v_k}(\bt(a)w_0 u^+ u^-)-W_{v_k'}(\bt(a)w_0 u^+ u^-))\psi_U^{-1}(u^+u^-)du^+ du^-.
\end{align*}

We have $$U_{w_0}^+=\wpair{\begin{pmatrix} 1&&\\ &u&\\ &&1 \end{pmatrix}, u\in U^{(n-1)}},$$
where $U^{(n-1)}$ is the upper triangular maximal unipotent of $\Sp_{2(n-1)}$. It is clear that for $u^+\in U_{w_0,m}^+$, we have
$$\bt(a)w_0u^+=u^+\bt(a)w_0.$$
Since $W_{v_k}( u^+\bt(a)w_0 u^-)=\psi_U(u^+)W_{v_k}(\bt(a)w_0 u^-)$, we then get

\begin{align*}
&W_{v_m}(\bt(a)w_0)-W_{v_m'}(\bt(a)w_0)\\
=&\frac{\vol(U_{w_0,m}^+)}{\vol(U_m)}\int_{U_{w_0,m}^-}(W_{v_k}(\bt(a)w_0 u^-)-W_{v_k'}(\bt(a)w_0  u^-))\psi_U^{-1}(u^-)du^-.
\end{align*}

By Corollary \ref{cor313}, we get
$$W_{v_k}(\bt(a)w_0u^-)-W_{v_k}(\bt(a)w_0u^-)=0,$$
for $u^-\in U_{w_0,m}^--U_{w_0,k}^-$. Thus we get
\begin{align*}
&W_{v_m}(\bt(a)w_0)-W_{v_m'}(\bt(a)w_0)\\
=&\frac{\vol(U_{w_0,m}^+)}{\vol(U_m)}\int_{U_{w_0,k}^-}(W_{v_k}(\bt(a)w_0 u^-)-W_{v_k'}(\bt(a)w_0  u^-))\psi_U^{-1}(u^-)du^-.
\end{align*}
By Lemma \ref{lemma22} or Eq.(\ref{eq21}), we get 
$$W_{v_k}(\bt(a)w_0 u^-)=\psi_U(u^-)W_{v_k}(\bt(a)w_0), W_{v_k'}(\bt(a)w_0 u^-)=\psi_U(u^-)W_{v_k'}(\bt(a)w_0) .$$
Thus we have 
\begin{align*}
&W_{v_m}(\bt(a)w_0)-W_{v_m'}(\bt(a)w_0)\\
=&\frac{\vol(U_{w_0,m}^+) \vol(U_{w_0,k}^-)}{\vol(U_m)}(W_{v_k}(\bt(a)w_0 u^-)-W_{v_k'}(\bt(a)w_0  u^-))\\
=&\frac{\vol(U_{w_0,k}^-)}{\vol(U_{w_0,m}^-)}(W_{v_k}(\bt(a)w_0 u^-)-W_{v_k'}(\bt(a)w_0  u^-))\\
=&q_F^{(2n-1)^2k-(2n-1)^2m} (W_{v_k}(\bt(a)w_0 u^-)-W_{v_k'}(\bt(a)w_0  u^-)).
\end{align*}
Plug this into Eq.(\ref{eq42}), we get
\begin{align}
&\Psi(W_{v_m},\phi_m, \tilde f_s^i)-\Psi(W_{v_m'},\phi_m,\tilde f_s^i)\label{eq43} \\
=&\gamma(\psi^{-1}) q_F^{(2n-1)^2k -(2n-1)m-i}\frac{\vol(D_m)}{\vol(U_{w_0,m}^-)} \int_{F^\times}  \left( W_{v_k}( \bt(a) w_0 ) -W_{v_k'}( \bt(a) w_0)  \right) \eta^{-1}_{-s-n}(a)da \nonumber \\
=&\gamma(\psi^{-1})q_F^{(2n-1)^2k-n^2m-3i} \int_{F^\times}  \left( W_{v_k}( \bt(a) w_0 ) -W_{v_k'}( \bt(a) w_0)  \right) \eta^{-1}_{-s-n}(a)da. \nonumber
\end{align}
By Eq.(\ref{eq41}, \ref{eq43}) and the local functional equation, we get
\begin{align}
&\gamma(s,\pi,\eta,\psi)-\gamma(s,\pi',\eta,\psi) \label{eq44}\\
=& \gamma(\psi^{-1}) q_F^{(2n-1)^2k}  \int_{F^\times}  \left( W_{v_k}( \bt(a) w_0 ) -W_{v_k'}( \bt(a) w_0)  \right) \eta^{-1}_{-s-n+1}(a)da. \nonumber
\end{align}

Now we can prove the main theorem. Note that $k$ only depends on the choices of  $v $ and $v'$, which are fixed at the begining. Since the function $a\mapsto W_{v_k}(\bt(a)w_0)$ and $a\mapsto W_{v_k'}(\bt(a)w_0)$ are continuous, we can take an integer $l=l(\pi,\pi')$ such that for $c\ge l$, we have
$$W_{v_k}(\bt(a_0a)w)=W_{v_k}(\bt(a)w_0), \textrm{ and }W_{v_k'}(\bt(a_0a))=W_{v_k'}(\bt(a)w_0),$$
for all $a_0\in 1+\CP^c$. Now it is clear that of $\eta$ is a quasi-character with $\cond(\eta)>l$, the right side of Eq. (\ref{eq44}) vanishes, and hence
$$\gamma(s,\pi,\eta,\psi)=\gamma(s,\pi',\eta,\psi).$$
\end{proof}

\noindent\textbf{Remark:} From Eq.(\ref{eq44}) and the Mellin inversion, we can get that if $\gamma(s,\pi,\eta,\psi)=\gamma(s,\pi',\eta,\psi)$ for all quasi-characters $\eta$ of $F^\times$, then $W_{v_k}(\bt(a)w_0)=W_{v_k'}(\bt(a)w_0)$ for all $a\in F^\times$. From this, it is easy to show that 
$$W_{v_m}(tw_0)=W_{v_m}(tw_0),$$
for all $t\in T$ and $m\ge 4^{l(w_0)^2}C$. By Theorem \ref{thm311}, we can get that $W_{v_m}(g)=W_{v_m'}(g)$ for all $g\in Bw_0B$. This should be the first step to get a local converse theorem for $\Sp_{2n}$ if Theorem \ref{thm311} works for all $w\in \bW$.\\

As a corollary of the stability of $\gamma$-factors, Theorem \ref{thm44}, and the multiplicativity of $\gamma$-factors \cite{Ka}, we have the following stable form for $\gamma(s,\pi,\eta,\psi)$.
\begin{prop}
Let $\pi$ be a generic representation of $\Sp_{2n}$ and let $\chi_1,\dots, \chi_n$ be non-trivial characters of $F^\times$. Then for sufficiently highly ramified character $\eta$ of $F^\times$, we have 
$$\gamma(s,\pi,\eta,\psi)=\gamma(s,\eta,\psi) \prod_{i=1}^n \gamma(s,\chi_i\eta, \psi)\gamma(s,\chi_i^{-1}\eta,\psi).$$
\end{prop}

\section{Howe vectors and stability of gamma factors for metaplectic groups}
In this section, we will extend the stability result to the $\widetilde \Sp_{2n}(F)$-case. Throughout this section, we assume that $F$ is a $p$-adic field with odd residue characteristic. 

 In $\widetilde \Sp_{2n}(F)$, we will frequently use the following relation, 
\begin{equation}\label{eq51}(p,\epsilon_1)(g,\epsilon)(p,\epsilon_1)^{-1}=(pgp^{-1},\epsilon),\end{equation}
for all $g\in \Sp_{2n}(F), p\in P$, and $\epsilon_1,\epsilon\in \mu_2=\wpair{\pm 1}$, see Eq.(2-6) of \cite{Sz1}. Recall that $P$ is the Siegel parabolic subgroup of $\Sp_{2n}(F)$.
 Denote $pr: \widetilde \Sp_{2n}(F)\ra \Sp_{2n}(F)$ the natural projection.

Let $K=\Sp_{2n}(\CO_F)$, which is a maximal open compact subgroup of $\Sp_{2n}(F)$. It is known that there is a group homomorphism $s:K\ra \widetilde \Sp_{2n}(F)$ such that $pr\circ s=\id_K$, see page 43 of \cite{MVW}. This splitting $s$ is known to be unique, see page 1662 of \cite{GS} for example. Denote the splitting $s$ by $s(k)=(k, \epsilon(k))$, where $\epsilon(k)\in \wpair{\pm 1}$. It is easy to see that the splitting over $K\cap U$ is also unique. In fact, any two such splittings differ by a quadratic character of $K\cap U$ and it suffices to show that $K\cap U$ has no nontrivial quadratic character. The latter statement follows from the fact that $2$ is a unit in $\CO_F^\times$ and (thus) the square map $K\cap U\ra K\cap U$ is surjective. Since there is a canonical splitting over $U$ given by $u\mapsto (u,1)$, it follows that $\epsilon(k)=1$ for all $k\in K\cap U$.

 Let $m$ be a positive integer and $K_m$ be the congruence subgroup $(1+\Mat_{2n\times 2n}(F))\cap \Sp_{2n}(F)\subset K$ as in $\S$2 and let $\tilde K_m$ be the inverse image of $K_m$ in $\widetilde \Sp_{2n}$. It is clear that $\tilde K_m=s(K_m)\times \mu_2$ as a group. Using the Iwahori decomposition, one can check that the square map $K_m\ra K_m$ is surjective and hence the splitting $s$ restricted to $K_m$ is also unique.
 
  Let $\psi$ be an unramified additive character of $F$, recall that we defined a character $\tau_m$ of $K_m$ in $\S$2. We now define a character $\tilde \tau_m$ of $\tilde K_m$ by
 \begin{equation}\tilde \tau_m((k,\epsilon(k) \epsilon))=\epsilon \tau_m(k), k\in K_m, \epsilon\in \wpair{\pm 1}.\end{equation}
 Since $s$ is a group homomorphism, it is clear that $\tilde \tau_m$ is indeed a character of $\tilde K_m$.
 
  Let $$d_m=\diag(\varpi^{-m(2n-1)}, \varpi^{-m(2n-3)},\dots, \varpi^{-m}, \varpi^m, \dots, \varpi^{m(2n-1)})\in \Sp_{2n}(F)$$
and $H_m=d_m K_m d_m^{-1}$ be as in $\S$2 and $\tilde d_m=(d_m,1)\in \widetilde \Sp_{2n}(F)$. Define a group homomorphism $s': H_m \ra \widetilde \Sp_{2n}(F)$, by $$s'(h)=\tilde d_m s(d_m^{-1}h d_m) \tilde d_m^{-1}.$$
We can check that $s'(d_m k d_m^{-1})=(d_m k d_m^{-1},\epsilon(k)), k\in K_m$. Since $s'$ is a group homomorphism, we have $c(d_mkd_m^{-1}, d_m k' d_m^{-1})=c(k,k')$. 

Let $\widetilde H_m=\tilde d_m \tilde K_m \tilde d_m^{-1}=s'(H_m)\times \mu_2$. We define a character $\tilde \psi_m$ on $\tilde H_m$ by 
\begin{equation}\tilde \psi_m( \tilde h)=\tilde \tau_m(\tilde d_m^{-1} \tilde h \tilde d_m).\end{equation}
If we write $\tilde h=s'(h)(1,\epsilon)$ for $h\in H_m$ and $\epsilon\in \mu_2$, we have
\begin{equation}\label{eq54}\tilde \psi_m(\tilde h)=\epsilon \psi_m(h).\end{equation}

\begin{lem}
We have $\tilde \psi_m|_{\tilde U_m}=\psi_{\tilde U}|_{\tilde U_m},$ where $\tilde U_m=\tilde U\cap \tilde H_m$ and $\psi_{\tilde U}$ is the generic character defined in $\S1$, i.e., $\psi_{\tilde U}(u,\epsilon)=\epsilon\psi_U(u)$.
\end{lem}
\begin{proof}
By the above discussion, we have $\epsilon(k)=1$ for $k\in K\cap U$. Since $\tilde U_m=\tilde H_m\cap \tilde U= \tilde d_m (s(K_m\cap U)\times \mu_2)\tilde d_m^{-1}$. A typical element $\tilde u\in \tilde U_m$ is of the form $\tilde d_m (u,\epsilon)\tilde d_m^{-1}=(d_m ud_m^{-1},\epsilon)$ by Eq.(\ref{eq51}), where $u\in K_m\cap U$ and $\epsilon\in \wpair{\pm 1}$. By the definition, we have $\tilde \psi_m(\tilde u)=\tilde \tau_m((u,\epsilon))=\epsilon \tau_m(u)=\epsilon \psi_m(d_m u d_m^{-1})$, and $\psi_{\tilde U}(\tilde u)=\epsilon \psi_{U}(d_m u d_m^{-1})$. Thus it suffices to show that $\psi_m(d_mud_m^{-1})=\psi_U(d_m ud_m^{-1})$. This is Lemma \ref{lemma21} (1).
\end{proof}
Note that the above calculation shows that $\tilde U_m=U_m\times \mu_2$ as a group and $\tilde \psi_m((u,\epsilon))=\epsilon \psi_m(u)$ for $u\in U_m, \epsilon\in \mu_2$.

Let $(\pi, V_\pi)$ be a genuine irreducible $\psi_{\tilde U}$-generic smooth representation of $\widetilde \Sp_{2n}$, and $v\in V_\pi$, we can define the Howe vector $v_m$ similarly, i.e., 
$$v_m=\frac{1}{\vol(\tilde U_m)}\int_{\tilde U_m}\tilde \psi_m(\tilde u)^{-1}\pi(\tilde u)vd\tilde u. $$
Note that $\tilde U_m=U_m \times \mu_2$ as a group,  $\tilde \psi_m((u,\epsilon))=\epsilon\psi_m(u)$ and $\pi((u,\epsilon))v=\epsilon\pi((u,1))v $, for $u\in U_m$ and $\epsilon\in \mu_2$, we get 
$$v_m=\frac{1}{\vol(U_m)}\int_{U_m}\psi_m(u)^{-1}\pi((u,1))vdu.$$
As in the $\Sp_{2n}$ case, we let $C=C(v)$ be a positive integer such that $v$ is fixed by $\pi(s(K_m))$. Then the analogue of Lemma \ref{lemma22} holds:
\begin{lem}\label{lemma52}
We have
\begin{enumerate}
\item $W_{v_m}(1)=1;$
\item if $m\ge C$, then $\pi(\tilde h)v_m=\tilde \psi_m(\tilde h)v_m$, for all $\tilde h\in \tilde H_m;$
\item for $k\le m$, we have 
$$v_m=\frac{1}{\vol(\tilde U_m)}\int_{\tilde U_m}\tilde \psi_m^{-1}(\tilde u) \pi(\tilde u)v_k d\tilde u=\frac{1}{\vol(U_m)}\int_{U_m} \psi_m(u)^{-1}\pi((u,1))v_k du.$$
\end{enumerate}
\end{lem}
\begin{proof}
(1) and (3) are clear and we show (2). Consider the vector $$\tilde v_m=\frac{1}{\vol(\tilde H_m)}\int_{\tilde H_m} \tilde \psi_m (\tilde h)^{-1}\pi(\tilde h)v d\tilde h.$$
It is clear that $\pi(\tilde h) \tilde v_m=\tilde \psi_m(\tilde h)\tilde v_m$. It suffices to show that $\tilde v_m=v_m$.

By Eq.(\ref{eq54}) and the fact that $\pi$ is genuine, we have
$$\tilde v_m=\frac{1}{\vol(H_m)}\int_{H_m} \psi_m(h)^{-1}\pi(s'(h))v dh.$$ 
From the Iwahori decomposition of $K_m$, we have $H_m=\bar B_m \cdot U_m$, where $\bar B_m=\bar B\cap H_m$, $U_m=U\cap H_m$ and $\bar B$ is the lower triangular Borel subgroup of $\Sp_{2n}(F)$. For $h=u \bar b\in H_m$ with $u\in U_m, \bar b\in \bar B_m$, we choose measures such that $dh=d\bar b du$. Thus
\begin{align*}
\tilde v_m=\frac{1}{\vol (H_m)}\int_{U_m} \int_{\bar B_m} \psi_m(u \bar b)^{-1} \pi(s'(u)s'(\bar b))v d\bar b du.
\end{align*}

By the definition of $\psi_m$, we can get $\psi_m(\bar b)=1$. Notice that $\bar b\in \bar B_m\subset H_m\cap K_m$, see \ref{lemma21} for example.  It is easy to see that the square map $H_m\cap K_m \ra H_m \cap K_m$ is surjective and thus there is a unique splitting over $H_m\cap K_m$, i.e., $s(\bar b)=s'(\bar b)$. For $m\ge C$, we have $\pi(s'(\bar b))v=\pi(s(\bar b))v=v$. Thus
\begin{align*}
\tilde v_m&=\frac{1}{\vol (U_m)}\int_{U_m}  \psi_m(u )^{-1} \pi(s'(u))v  du\\
&=\frac{1}{\vol(U_m)}\psi_m(u)^{-1}\pi((u,1))v du\\
&=v_m.
\end{align*}
This proves (2).
\end{proof}
Using relation Eq.(\ref{eq51}), one can check that all of the results in $\S$2-4 for $\Sp_{2n}$ have corresponding analogue for $\widetilde \Sp_{2n}$ with similar proof. We only check one of them to illustrate the idea how to modify the proof in the $\Sp_{2n}$ case so that it is adapt to the $\widetilde \Sp_{2n}$ case.

\begin{lem}
For $t\in T$ the torus of $\Sp_{2n}$ and $m\ge C$. If $\epsilon\in \mu_2$ and $W_{v_m}((t,\epsilon))\ne 0$, then $\alpha_i(t)\in 1+\CP^m$ for $1\le i\le n-1$ and $\beta(t)\in 1+\CP^m$.
\end{lem}
\begin{proof}
This is the analogue of Lemma \ref{lemma23}, and in fact similar proof goes through. For a simple root $\gamma$, we take $r\in \CP^{-m}$ so that $\bx_{\gamma}(r)\in H_m$. Then $(\bx_{\gamma}(r),1)\in \tilde H_m$. By Eq.(\ref{eq51}), we have
$$(t,\epsilon) (\bx_{\gamma}(r),1)(t,\epsilon)^{-1}=(t\bx_{\gamma}(r)t^{-1},1)=(\bx_{\gamma}(\gamma(t)r),1),$$
see the proof of Lemma \ref{lemma23}. By Lemma \ref{lemma52}, we have
$$\tilde \psi_m((\bx_{\gamma}(r),1))W_{v_m}((t,\epsilon))=\psi_{\tilde U} ((\bx_{\gamma}(\gamma(t)r),1))W_{v_m}((t,\epsilon)). $$
By the definition of $\tilde \psi_m$ and $\psi_{\tilde U}$, if $W_{v_m}((t,\epsilon))\ne 0$, we have $\psi(r)=\psi(\gamma(t)r)$ for all $r\in \CP^{-m}$. Since $\psi$ is unramified, we get $\gamma(t)\in 1+\CP^m$.
\end{proof}

Similar consideration as in the $\Sp_{2n}$ case will give us the stability of gamma factors for $\widetilde \Sp_{2n}$, i.e, we have the following 
\begin{thm}
Let $(\pi,V_{\pi})$ and $(\pi',V_{\pi'})$ be two genuine irreducible smooth $\psi_{\tilde U}$-generic representation of $\widetilde \Sp_{2n}(F)$ with the same central character, where $F$ is a $p$-adic field such that its residue characteristic is not $2$. If $\eta$ is a highly ramified quasi-character of $F^{\times}$, then
$$\gamma(s,\pi,\eta,\psi)=\gamma(s,\pi',\eta,\psi).$$
\end{thm}

\section{Stability of $\gamma$-factors for $\RU_{E/F}(n,n)$}
 In this section, we assume $E/F$ is a quadratic extension of $p$-adic fields, and denote $x\mapsto \bar x$ the nontrivial Galois action in $\Gal(E/F)$. For objects corresponding to $E$, we will add a subscript $E$. For example, we denote $\CO$ the integer ring of $F$ (as in the notation section) and $\CO_E$ the integer ring of $E$.
 
  The group $\RU_{E/F}(n,n)$ is defined by 
 $$\RU_{E/F}(n,n)=\wpair{g\in \GL_{2n}(E)| {}^t\! \bar g \begin{pmatrix}& J_n\\ -J_n & \end{pmatrix} g=  \begin{pmatrix}& J_n\\ -J_n & \end{pmatrix} },$$
where $J_n$ is the same as in the notation section. We will use similar notations as in the $\Sp_{2n}$ case. For example, 
$$\bm_n(g)=\begin{pmatrix}g& \\ & g^* \end{pmatrix}, g\in \GL_n(E), g^*=J_n {}^t\! \bar g^{-1}J_n,$$
 $$\bn_n(b)=\begin{pmatrix}I_n & b\\ & I_n \end{pmatrix}, b\in \Mat_{n\times n}(E), {}^t\! \bar b=J_n bJ_n,$$
 and $$r(y,x)=\bm_n \begin{pmatrix}I_{n-2}&&y \\ &1&x \\ &&1 \end{pmatrix}, y\in \Mat_{(n-2)\times 1}(E), x\in E.$$
 Let $M$ and $N$ be the subgroup which consists elements of the form $\bm_n(g)$ and $\bn_n(b)$ respectively. Let $U_M$ be the standard maximal unipotent subgroup of $M$. Let $U=U_M\ltimes N$, which is a maximal unipotent subgroup of $\RU_{E/F}(n,n)$. Let $B$ be the standard upper triangular Borel subgroup and $B=TU$ is the Levi decomposition.
 
  Let $\psi_E$ (resp. $\psi$) be a nontrivial additive character of $E$ (resp. $F$).  We consider the generic character $\psi_U$ on $U$ defined by 
 $$\psi_U|_{U_M}((u_{ij}))=\psi_E(\sum_{i=1}^{n-1} u_{i,i+1}), (u_{ij})\in U_M,$$
 and 
 $$\psi_U|_{N}(u_{ij})=\psi(u_{n,n+1}), (u_{ij})\in N.$$
 
Let $w_1$ be the same as in the $\Sp_{2n}$ case and define $j(g)=w_1 gw_1^{-1}$ for $g\in \RU_{E/F}(n,n)$. 

\subsection{Weil representations and induced representations on $\RU_{E/F}(1,1)$}
The group $\RU_{E/F}(1,1)$ can be viewed as a subgroup of $\Sp_4$. Let $\mu$ be a character of $E^\times$ such that $\mu|_{F^\times}$ is the class field theory character on $F^\times$ defined by $E/F$. Then it is know that $\mu$ defines a splitting $s_\mu: \RU_{E/F}(1,1)\incl \widetilde \Sp_4$ of the double cover map $\widetilde \Sp_4\ra \Sp_4$ over $\RU_{E/F}(1,1)$. Thus for a nontrivial additive character $\psi$ of $F$, we have a Weil representation $\omega_{\mu,\psi}$ of $\RU_{E/F}(1,1)$ on the space $\CS(E)$. For the splitting $s_\mu$ and the Weil representation, see \cite{HKS} for example. 

Given a quasi-character $\eta$ of $E^\times$ and $s\in \BC$, we can consider the induced representation $I(s,\eta)=\Ind_{B^1}^{\RU_{E/F}(1,1)}(\eta_{s-1/2})$ of $\RU_{E/F}(1,1)$, where $B^1$ is the upper triangular Borel subgroup of $\RU_{E/F}(1,1)$. By \cite{Ba2}, we can parametrize the space $I(s,\eta)$ using the space $\CS(F)$, like the $\GL_2$ case. 
\subsection{Local zeta integrals and $\gamma$-factors} For simplicity, we denote $G_n=\RU_{E/F}(n,n)$. 
We consider the embedding $\iota: G_1\incl G_n$
$$ g\mapsto \begin{pmatrix} I_{n-1}&& \\ & g& \\ && I_{n-1}\end{pmatrix}.$$
We will not distinguish an element $g\in G_1$ with its image $\iota(g)\in G_n$. 

Let $\pi$ be a $\psi_U$-generic representation of $G_n=\RU_{E/F}(n,n)$, and $\eta$ be a quasi-character of $E^\times$. For $W\in \CW(\pi,\psi_U), \phi\in \CS(E), f\in I(s,\eta)$, we consider the local zeta integral
$$\Psi(W, \phi, f_s)=\int_{U^1\setminus G_1} \int_{E^{(n-2)}}\int_E W(j(r(y,x)g) ) \omega_{\mu,\psi^{-1}}(g)\phi(x)f_s(g)dx dy dg.$$
There is a standard intertwining operator $M_s: I(s,\eta)\ra I(1-s, \bar \eta^{-1})$. 
\begin{prop}\label{prop61}
The local zeta integral $\Psi(W,\phi,f_s)$ is absolutely convergent for $\Re(s)>>0$ and defines a rational function of $q_E^{-s}$. Moreover, there exists a rational function $\gamma(s,\pi,\eta,\psi)$ such that
$$\Psi(W,\phi, M_s(f_s))=\gamma(s,\pi,\eta,\mu,\psi)\Psi(W,\phi,f_s),$$
for all $W\in \CW(\pi,\psi_U), \phi\in \CS(E)$ and all $f_s\in I(s,\eta)$.
\end{prop}
\begin{proof}
The convergence of the local zeta integral is standard, which comes from a standard gauge estimate of $W$. The existence of the $\gamma$ factors comes from the the uniqueness of the Fourier-Jacobi models in the unitary case, see \cite{GaGP, Su}. We omit the details.
\end{proof}

\noindent \textbf{Remark:} The local integrals and the $\gamma$-factors in the unitary group case are analogues in the $\Sp_{2n}$ case. But to the author's knowledge, the local theory in the unitary group case is not studied in the literature.

\subsection{Howe vectors}
We can define the Howe vectors similarly. We provide a little bit details. Let $K_m=(1+\Mat_{2n\times 2n}((\CP\CO_E)^m))$, where $\CP$ is the maximal ideal of $F$ and $\CO_E$ is the ring of integers of $E$. Note that if $E/F$ is unramified, then $\CP\CO_E=\CP_E$, the maximal ideal of $E$. If $E/F$ is ramified, then $\CP\CO_E=\CP_E^2$. Let $d_m=\diag(\varpi^{-(2n-1)m}, \varpi^{-(2n-3)m}, \dots, \varpi^{-m}, \varpi^m, \dots, \varpi^{(2n-3)m},\varpi^{(2n-1)m}),$ where $\varpi$ is a uniformizer in $F$, and define 
$$H_m=d_m K_m d_m^{-1}.$$

Assume $\psi$ and $\psi_E$ are unramified additive characters of $F$ and $E$ respectively. We can define a character $\psi_m$ of $H_m$ similar to the $\Sp_{2n}$ case. Let $(\pi,V_\pi)$ be a $\psi_U$-generic irreducible smooth representation of $G_n$. For $v\in V_\pi$ with $W_v(1)=1$, we define
$$v_m=\frac{1}{\vol(U_m)}\int_{U_m} \psi_U(u)^{-1}\pi(u)v du.$$

Let $C$ be an integer such that $v$ is fixed by $\pi(K_C)$. Then the counterpart of Lemma \ref{lemma22} also holds in our case. 

The counterpart of Lemma \ref{lemma23} becomes
\begin{lem}\label{lemma62}
Let $m\ge C$ and $t\in T$. If $W_{v_m}(t)\ne 0$, then $\alpha_i(t)\in 1+(\CP\CO_E)^m$ and $\beta(t)=1+\CP^m$.
\end{lem}

Denote $E^1$ the norm 1 elements in $E^\times$. The counterpart of Lemma \ref{lemma24} has the following form
\begin{lem}\label{lemma63}
Suppose that $E/F$ is unramified, or $E/F$ is ramified but the residue characteristic of $F$ is not $2$. For $a\in E^\times$, if $a\bar a\in 1+\CP^m$, then $a\in E^1(1+(\CP\CO_E)^m)$.
\end{lem}
\begin{proof}
If $E/F$ is unramified, the result follows from the fact that the norm map $1+\CP_E^m\ra 1+\CP_E^m, a\mapsto a\bar a$ is surjective, see Proposition 3, Chapter V, $\S$2 of \cite{Se}. If $E/F$ is ramified, it needs a little bit more work. See Lemma 3.3 of \cite{Zh1} for more details.
\end{proof}

Now we fix two $\psi_U$-generic irreducible smooth representations $(\pi,V_\pi)$ and $(\pi',V_{\pi'})$ with the same central character. We fix $v\in V_{\pi},v'\in V_{\pi'}$ such that $W_{v}(1)=W_{v'}(1)=1$, and an integer $C$ such that $v$ is fixed by $\pi(K_C)$ and $v'$ is fixed by $\pi'(K_C)$. 

The counterpart of Lemma \ref{lemma312} also holds:
\begin{lem}\label{lemma54}
\begin{enumerate}
\item If $E/F$ is unramified, or $E/F$ is ramified but the residue characteristic of $F$ is not $2$, then 
$$W_{v_m}(g)=W_{v_m'}(g),$$
for all $g\in B,m\ge C$.
\item We have $W_{v_m}(tw)=0=W_{v_m'}(tw)$ for all $w<w_0, t\in T$ and $m\ge C$.
\end{enumerate}
\end{lem}
\begin{proof}
Notice that $T_m=T\cap H_m= \diag(1+(\CP\CO_E)^m,\dots, 1+(\CP\CO_E)^m)$. A simple calculation as in the $\Sp_{2n}$ case shows (1) following Lemma \ref{lemma62} and Lemma \ref{lemma63}. 

The proof of (2) is the same as in the $\Sp_{2n}$ case.
\end{proof}

Note that Theorem \ref{thm311} and hence Corollary \ref{cor313} also holds in the unitary group case. The same calculation as in $\S$4 will show the stability of $\gamma$-factors in the unitary case. More precisely, we have
\begin{thm}
Suppose that $E/F$ is unramified, or $E/F$ is ramified but the residue characteristic of $F$ is not $2$. Let $\pi,\pi'$ be two $\psi_U$-generic irreducible smooth representations of $\RU_{E/F}(n,n)$. Then if $\eta$ is a highly ramified quasi-character of $F^\times$, we have
$$\gamma(s,\pi,\eta,\mu,\psi)=\gamma(s,\pi',\eta,\mu,\psi).$$
\end{thm}

\end{document}